\documentclass{article}
\usepackage[utf8]{inputenc}
\usepackage{mathtools, nccmath}
\mathtoolsset{showonlyrefs}  
\usepackage{geometry}
\usepackage{xcolor}
\usepackage{amsmath}
\usepackage{multirow}
\usepackage{amsthm}
\usepackage{changepage}
\usepackage{amssymb}
\usepackage{bbm}
\usepackage{xcolor}
\usepackage{graphicx}
\usepackage{algorithm,algorithmic}
\usepackage{cases}
\usepackage{enumitem}
\usepackage{subfig}
\usepackage[thinlines]{easytable}
\usepackage[export]{adjustbox}
\graphicspath{ {./images/} }
\usepackage[english]{babel}

\newtheorem{theorem}{Theorem}[section]

\newtheorem{lemma}[theorem]{Lemma}
\newtheorem{proposition}[theorem]{Proposition}
\newtheorem{assumption}[theorem]{Assumption}
\newtheorem{remark}[theorem]{Remark}

\DeclareMathOperator{\proj}{Proj}
\DeclareMathOperator{\midd}{Center}

\DeclareFontFamily{U}{mathx}{\hyphenchar\font45}
\DeclareFontShape{U}{mathx}{m}{n}{
      <5> <6> <7> <8> <9> <10>
      <10.95> <12> <14.4> <17.28> <20.74> <24.88>
      mathx10
      }{}
\DeclareSymbolFont{mathx}{U}{mathx}{m}{n}
\DeclareFontSubstitution{U}{mathx}{m}{n}
\DeclareMathAccent{\widecheck}{0}{mathx}{"71}
\DeclareMathAccent{\wideparen}{0}{mathx}{"75}

\usepackage[colorlinks,citecolor=blue]{hyperref}

\title{Convergence of Actor-Critic Learning for Mean Field Games and Mean Field Control in Continuous Spaces}

\author{
  Jean-Pierre Fouque\thanks{Department of Statistics and Applied Probability, South Hall, University of California, Santa Barbara, CA 93106, USA (E-mail: \href{mailto:fouque@pstat.ucsb.edu}{fouque@pstat.ucsb.edu}). Work supported by NSF grant DMS-1953035.} 
  \and Mathieu Lauri\`ere\thanks{NYU Shanghai Center for Data Science; NYU-ECNU Institute of Mathematical Sciences at NYU Shanghai; NYU Shanghai, 567 West Yangsi Road, Shanghai, 200126, People’s Republic of China (E-mail: 
  \href{mailto:mathieu.lauriere@nyu.edu}{mathieu.lauriere@nyu.edu}).}
  \and  Mengrui Zhang\thanks{Department of Statistics and Applied Probability, South Hall, University of California, Santa Barbara, CA 93106, USA (E-mail: \href{mailto:mengrui@umail.ucsb.edu}{mengrui@umail.ucsb.edu}).
  }}
\renewcommand{\arraystretch}{1.5}
\begin{document}

\maketitle
\centerline{\textbf{Abstract}}
\begin{adjustwidth}{50pt}{50pt} 
We establish the convergence of the deep actor-critic reinforcement learning algorithm presented in \cite{angiuli2023deep} in the setting of continuous state and action spaces with an infinite discrete-time horizon. This algorithm provides solutions to Mean Field Game (MFG) or Mean Field Control (MFC) problems depending on the ratio between two learning rates: one for the value function and the other for the mean field term. In the MFC case, to rigorously identify the limit, we introduce a discretization of the state and action spaces, following the approach used in the finite-space case in \cite{Andrea23}. The convergence proofs rely on a generalization of the two-timescale framework introduced in \cite{Borkar97}. We further extend our convergence results to Mean Field Control Games, which involve locally cooperative and globally competitive populations. Finally, we present numerical experiments for linear-quadratic problems in one and two dimensions, for which explicit solutions are available.
\end{adjustwidth}

\section{Introduction}

Reinforcement Learning (RL) is a branch of machine learning in which an autonomous agent learns to make sequential decisions through interaction with an environment, guided by feedback in the form of rewards or penalties. This process is typically formalized within the mathematical framework of Markov Decision Processes (MDPs)~\cite{sutton2018reinforcement}. Over the past decade, RL has gained significant traction, leading to remarkable advances across diverse domains—from mastering classical games such as Atari~\cite{mnih2015human} and Go~\cite{silver2016mastering}, to providing sophisticated control in robotics~\cite{gu2017deep,vecerik2017leveraging}, and more recently, to reinforcement learning from human feedback~\cite{ouyang2022training}.

While classical RL primarily addresses the single-agent setting, multi-agent reinforcement learning (MARL) extends the paradigm to scenarios where multiple agents learn concurrently while interacting. Foundational surveys~\cite{busoniu2008comprehensive,zhang2021multi} and game-theoretic treatments~\cite{lanctot2017unified,yang2020overview} highlight the richness of this area. However, despite notable successes—particularly in competitive or cooperative games—model-free MARL continues to face scalability challenges as the number of agents increases.

In parallel, the theory of mean field games (MFGs), introduced independently by Lasry and Lions~\cite{MR2295621} and by Caines, Huang, and Malhamé~\cite{MR2346927}, has emerged as a powerful framework for analyzing games with a very large number of symmetrically interacting agents. By focusing on the collective behavior of the population through a representative agent model, MFGs offer a tractable approximation of Nash equilibria that would otherwise be intractable in many-player games. While the classical formulation centers on Nash equilibria, the related concept of mean field control (MFC) shifts attention to socially optimal settings in which agents coordinate to minimize a global cost~\cite{MR3134900,carmona2018probabilisticI-II}.

In recent years, there has been growing interest in leveraging model-free RL to solve MFG and MFC problems. Several algorithmic approaches have been proposed (see \cite{lauriere2022learning} for a survey), though most focus on only one of the two problem classes. For MFGs, two main strands of methods have emerged: those exploiting strict contraction and fixed-point iteration schemes—often implemented via tabular or deep Q-learning~\cite{guo2019learning,cui2021approximately,anahtarci2023q}—and those relying on monotonicity properties, such as fictitious play~\cite{elie2020convergence,perrin2020continuousfp}. \cite{lauriere2022scalable} combined fictitious play with deep RL to solve
finite space MFGs, and~\cite{magnino2025solving} extended the method to continuous space MFGs with possibly local dependence on the distribution. Two-timescale learning strategies have also been applied to MFGs; see, for instance,~\cite{mguni2018decentralised,SubramanianMahajan-2018-RLstatioMFG}.

On the MFC side, the problem can be reformulated as a mean field Markov decision process (MF-MDP), where the mean field becomes part of the state description. This formulation has been explored through value-based methods such as Q-learning~\cite{carmona2023model,gu2021meanQ}, as well as through policy gradient~\cite{CarmonaLauriereTan-2019-LQMFRL}, actor-critic~\cite{frikha2025actor}, and model-based approaches~\cite{pasztor2021efficientmodelbased}. %

The present paper builds on the unified framework proposed in~\cite{Andrea20}, which introduced a two-timescale RL algorithm capable of solving both MFG and MFC problems by appropriately tuning the relative learning rates. This approach jointly updates the estimated population distribution and the representative agent’s value function, enabling convergence to the desired equilibrium—whether Nash (MFG) or socially optimal (MFC)—within a single methodological framework.

In Section~\ref{sec:MF}, we introduce the framework for discrete-time, infinite-horizon mean field problems in continuous spaces. Section~\ref{sec:ClassicalRL} provides a brief review of the actor-critic RL algorithm. The actor-critic algorithms for MFG and MFC problems are presented in Section~\ref{sec:algos}, where we emphasize the differences between the two problems and the introduction of bins (discretization) in the MFC case.
Our convergence results are derived in Section~\ref{sec:problems}, starting with the idealized deterministic problems introduced in Section~\ref{sec:idealized}. The main convergence result for the idealized MFG algorithm is stated in Theorem~\ref{mainThMFG}, and the corresponding result for the idealized MFC algorithm appears in Theorem~\ref{mfc_converge}. The extension to the full asynchronous algorithms with stochastic approximation, introduced in Section~\ref{sec:algos}, is presented in Section~\ref{sec:fullalgo}.
Numerical results for linear-quadratic models are presented in Section~\ref{sec:Num}, beginning with the one-dimensional case in Section~\ref{sec:Num1d} and followed by the two-dimensional case in Section~\ref{sec:Num2d}. In the latter, the introduction of bins in the MFC algorithm leads to a significant improvement in identifying the limiting population distribution (Figure~\ref{fig:Ex_Im6}) compared with the results in~\cite{angiuli2023deep}.
Finally, in Section~\ref{sec:MFCG}, we present a unified actor-critic RL algorithm for Mean Field Control Games, which describe competitive interactions among large numbers of cooperative groups of agents, as introduced in~\cite{10.1145/3533271.3561743} for the discrete-space case.

\section{Asymptotic Infinite Horizon Mean Field Problems}\label{sec:MF}

In this section, we introduce the framework for MFG and MFC problems in the discrete-time, infinite-horizon setting. We further emphasize the mathematical distinction between the two classes of mean-field problems, highlighting that they yield distinct solutions despite the apparent similarities in their formulation.

In both cases, the mathematical setting is
\begin{itemize} 
    \item a measurable function $f$: $\mathbbm{R}^d\times\mathcal{P}(\mathbbm{R}^d)\times\mathbbm{R}^k\to\mathbbm{R}$ known as the running cost denoted by $f(x,\mu,a)$, where $x$ is the individual state, $a$ is the individual action, $\mu$ represents the mean field, $\mathcal{P}(\mathbbm{R}^d)$ is the space of probability measures on $\mathbbm{R}^d$, and $\mathbbm{R}^k$ is the set of actions,
    \item a discount factor $0<\gamma < 1$ which measures the relative importance of future rewards or costs compared to immediate ones,
    \item a discrete-time Markov chain $(X^{\pi,\mu}_n)_{n \ge 0}$ in $\mathbbm{R}^d$ with transition density $p(x,x',\mu,a)$ which depends on the population distribution $\mu \in  \mathcal{P}(\mathbbm{R}^d)$ and the policy $\pi \in \mathcal{P}(\mathbbm{R}^k)$, according to which actions are sampled.
\end{itemize}

Further smoothness assumptions on $f$ and $p$ will be specified later. In particular, we will assume that the policy $\pi$ is in an admissible set $\mathcal{A}$, such that the Markov chain $(X^{\pi,\mu}_n)_{n \ge 0}$ is ergodic for all $\mu \in \mathcal{P}(\mathbbm{R}^d)$.  

In our numerical examples in Section~\ref{sec:Num}, the dynamics will come from the time-discretization of an $\mathbbm{R}^d$-valued controlled diffusion process $(X_t)_{t \ge 0}$, where for the discretization $0<t_1<t_2<\cdots <t_n<\cdots$ we use $X_n$ to represent $X_{t_n}$.

Since our actor-critic reinforcement learning algorithms use policies as actors, next we present the MFG and MFC problems respectively as the search for an equilibrium policy (MFG) and an optimal policy (MFC). Note that we consider infinite-horizon asymptotic MFG and MFC problems, as introduced in \cite{Andrea20} for the finite-space case.

\subsection{Mean Field Game}
\label{sec:mfgdef}

The solution of a mean field game (MFG), referred to as an MFG equilibrium, is a policy--mean field pair
$$(\hat\pi,\hat\mu)\in {\cal A}\times\mathcal{P}(\mathbbm{R}^d),$$
satisfying the following conditions:
\begin{enumerate}
    \item \textbf{Optimality.} For a fixed probability distribution $\hat\mu$, $\hat\pi$ minimizes the stochastic control cost functional
    $$
    \inf_{\pi\in{\cal A}} J_{\hat\mu}(\pi)
    = \inf_{\pi\in{\cal A}} \mathbbm{E}\left[\sum_{n=0}^\infty \gamma^{n} f\left(X_{n}^{\pi,{\hat\mu}},\hat\mu, A_n\right)\right],
    $$
    where the process $(X^{\pi,{\hat\mu}}_{n})_{n\geq 0}$ follows the dynamics
\begin{align*}
    \begin{cases}
        X^{\pi,{\hat\mu}}_{0} \sim \hat\mu_0, \\[4pt]
        P(X^{\pi,{\hat\mu}}_{n+1} \in \mathcal{B} \mid X^{\pi,{\hat\mu}}_{n}=x, A_{n}=a, \hat\mu) 
        = \int_{\mathcal{B}} p(x,x',{\hat\mu},a)\,dx', \\[4pt]
        A_{n} \sim \pi(\cdot|X_{n}^{\pi,\hat\mu}) \quad \text{where actions are sampled independently at each time step } n\geq 0.
    \end{cases}
\end{align*}

\item \textbf{Fixed point.} The probability distribution $\hat\mu$ satisfies the fixed point condition
$$
\hat\mu = \lim_{n\to\infty}\mathcal{L}(X_{n}^{\hat\pi,\hat\mu}),
$$
where $\mathcal{L}(X_{n}^{\hat\pi,\hat\mu})$ denotes the law of $X_{n}^{\hat\pi,\hat\mu}$, assuming the limit exists.
\end{enumerate}

To ensure the above problem is well-defined, we restrict admissible policies $\pi$ to those for which, for any $\mu$, the controlled process $(X^{\pi,\mu}_{n})$ is ergodic, so that in particular the limit $\lim_{n\to\infty} \mathcal{L}(X^{\pi,\mu}_{n})$ exists. 
This property is made precise in Assumption \ref{ACM1}.

\subsection{Mean Field Control}
\label{sec:mfcdef}

The solution of a mean field control (MFC) problem is a policy $\pi^* \in \mathcal{A}$ that minimizes an optimal control problem with McKean--Vlasov (MKV) dynamics:
$$
\inf_{\pi\in\mathcal{A}} J(\pi)
= \inf_{\pi\in\mathcal{A}} \mathbbm{E}\left[\sum_{n=0}^\infty \gamma^{n} f\left(X_{n}^{\pi},\mu^\pi, A_n\right)\right],
$$
where the process $(X^{\pi}_{n})_{n\geq 0}$ follows the dynamics
\begin{align*}
    \begin{cases}
        X^{\pi}_{0} \sim \mu_0, \\[4pt]
        P(X^{\pi}_{n+1} \in \mathcal{B} \mid X^{\pi}_{n}=x, A_{n}=a, \mu^\pi) 
        = \int_{\mathcal{B}} p(x,x',\mu^\pi,a)\,dx', \\[4pt]
        A_{n} \sim \pi(\cdot|X_{n}^{\pi}) \quad \text{where actions are sampled independently at each time step } n \ge 0,
    \end{cases}
\end{align*}
and where $\mu^\pi = \lim_{n\to\infty}\mathcal{L}(X_{n}^\pi)$ denotes the limiting distribution of the state process. 
We denote by $\mu^* := \mu^{\pi^*}$ the limiting mean field distribution corresponding to the optimal control. 

To ensure the above problem is well-defined, we restrict admissible policies $\pi \in \mathcal{A}$ to those for which the process $(X^{\pi}_{n})$ admits a limiting distribution, that is, $\lim_{n\to\infty} \mathcal{L}(X^{\pi}_{n})$ exists.

\subsection{Mean Field Game/Control Distinction}

We now summarize the key mathematical distinction between MFG and MFC problems. In the former, an optimal control problem is solved for a fixed distribution $\mu$, and the corresponding mean field $\hat\mu$ is then recovered as a fixed point. If we define the operator 
$$ 
    \Phi:\mathcal{P}(\mathbb{R}^d)\to\mathcal{P}(\mathbb{R}^d), \qquad
    \Phi(\mu)=\lim_{n\to\infty}\mathcal{L}(X_n^{\tilde{\pi},\mu}),
$$
where $\tilde{\pi}= \arg\min J_{\mu}(\pi)$, then the MFG equilibrium arises as a fixed point of $\Phi$ in the sense that 
$$
    \hat\mu=\Phi(\hat\mu).
$$

In the latter case, the mean field is explicitly the law of the state process throughout the optimization, and the problem can be viewed as a control problem in which the state distribution directly influences the dynamics. This distinction has been well documented in the literature; see, for example,~\cite{carmona2013control,carmona2015mean,carmona2019price,cardaliaguet2019efficiency,carmona2023nash}. In the MFC case, the distribution $\mu^\pi$ evolves with the choice of policy $\pi$, whereas in the MFG case it remains fixed during the optimization step and is subsequently updated through a fixed-point condition. These interpretations play a key role in guiding the development of our algorithms.

In general, 
$$
    (\hat\pi,\hat\mu)\neq(\pi^*,\mu^*),
$$
for the same choice of running cost, discount factor, and state dynamics. This distinction manifests in the RL algorithms we design for MFG and MFC problems. In the MFG setting, the limit is identified as a solution of a Bellman equation, whereas in the MFC setting, the Bellman equation depends explicitly on both the population distribution and the policy.

\section{Classical Actor-Critic Reinforcement Learning}\label{sec:ClassicalRL}

\subsection{A Quick Review of Classical RL}

Classical reinforcement learning (RL) aims to solve a Markov Decision Process (MDP) as follows. At each discrete time $n$, the agent observes her state $X_n$ and chooses an action $A_n$ based on it. The environment then transitions to a new state $X_{n+1}$ and provides a reward $r_{n+1}$ to the agent. The goal of the agent is to find an optimal policy $\pi$ that assigns, for each state, a probability distribution over actions in order to maximize the expected cumulative reward
$$
    \mathbbm{E}\left[\sum_{n=0}^{\infty}\gamma^n r_{n+1}\right],
$$
where $r_{n+1} = r(X_n,A_n)$ is the instantaneous reward, $\gamma\in(0,1)$ is a discounting factor, and $X_{n+1}$ is distributed according to a transition probability that depends on $X_n$ and the action drawn from $\pi(X_n)$.

A policy $\pi: \mathcal{X} \to \mathcal{P}(\mathcal{A})$ specifies, for each state, a probability distribution over actions. The agent’s goal is then to find an optimal policy $\pi^*$ satisfying
$$
    \pi^* \in \arg\max_{\pi} \mathbbm{E}_{\pi}\left[\sum_{n=0}^{\infty}\gamma^n r_{n+1}\right].
$$
Since the reward $r_{n+1} = r(X_n, A_n)$ depends on the current state and action, the expected cumulative reward depends on the chosen policy $\pi$.

For a given policy $\pi$, the state-value function $v_\pi: \mathcal{X} \to \mathbb{R}$ is defined by
$$
    v_\pi(x) = \mathbbm{E}_{\pi}\left[\sum_{n=0}^{\infty}\gamma^n r(X_n, A_n) \,\middle|\, X_0 = x\right].
$$

\subsection{Temporal Difference Method}

The value function satisfies the Bellman equation, which relates the value of the current state to that of the next state:
$$
    v_\pi(x) = \mathbb{E}_{\pi}\big[r_{n+1} + \gamma v_\pi(X_{n+1}) \,\big|\, X_n = x \big].\label{bell}
$$
Because the state transitions are Markovian, Equation~\eqref{bell} holds for all $n \ge 0$, not only at the initial state. Importantly, for a given policy $\pi$, the Bellman equation uniquely determines the value function, an assumption we adopt in the following.

Temporal Difference (TD) methods iteratively update an approximation $V$ of $v_\pi$ to drive the TD error $\delta_n$ to zero at each timestep:
$$
    \delta_n := r_{n+1} + \gamma V(X_{n+1}) - V(X_n),
$$
These TD methods require only the immediate transition sequence $\{X_n, r_{n+1}, X_{n+1}\}$ and do not rely on knowledge of the MDP model.

When a population distribution is involved, the Bellman equations for MFG and MFC take different forms. In MFG, the population distribution is fixed when computing the value function, whereas in MFC, the distribution may evolve with each state-action pair $\{X_n, A_n\}$, as it directly influences the dynamics. Further details will be provided in subsequent sections.

\subsection{Classical Actor-Critic Algorithm}

In the classical actor-critic algorithm for continuous action spaces, the policy is often modeled as a Gaussian distribution:
\begin{equation}
    \pi_\psi(a\mid x) = \mathcal{N}\big(a;\, \mu_\psi(x),\, \Sigma_\psi(x)\big), \quad a \in \mathcal{A}, \; x \in \mathcal{X},
\end{equation}
for which $\nabla_\psi \log \pi_\psi(a\mid x)$ is available in closed form, enabling stable and differentiable updates. This framework couples value estimation with policy improvement, leveraging the sample efficiency of TD learning while enabling gradient-based policy optimization.

The classical actor-critic algorithm also uses the temporal difference method with a differentiable stochastic policy $\pi_\psi$ and a parametric value function $V_\theta(x)$. The performance objective is
\begin{equation}
    J(\psi) = \mathbb{E}_{\pi_\psi}\Big[ \sum_{n=0}^\infty \gamma^n\, r(X_n,A_n) \Big].
\end{equation}
The critic estimates $V^{\pi_\psi}$ via temporal-difference (TD) learning. At time $n$, it forms the TD error
\begin{equation}
    \delta_n = r_n + \gamma\,V_\theta(X_{n+1}) - V_\theta(X_n),
\end{equation}
and performs the semi-gradient update
\begin{equation}
    \theta_{n+1} = \theta_n + \rho^V_n\, \delta_n\, \nabla_\theta V_\theta(X_n).
\end{equation}

The actor applies the policy-gradient theorem,
\begin{equation}
    \nabla_\psi J(\psi) = 
    \mathbb{E}_{\pi_\psi}\big[\nabla_\psi \log \pi_\psi(A_n \mid X_n)\; Q^{\pi_\psi}(X_n,A_n)\big],
\end{equation}
and replaces the action-value function with a low-variance advantage estimate from the critic, e.g., $A_n \approx \delta_n$, to update
\begin{equation}
    \psi_{n+1} = \psi_n + \rho^\Pi_n\, \delta_n\, \nabla_\psi \log \pi_\psi(A_n \mid X_n).
\end{equation}

Under a two-time-scale scheme in which the critic tracks faster than the actor,
\begin{equation}
    \frac{\rho^\Pi_n}{\rho^V_n}\to 0,\qquad 
    \sum_{n}\rho^\Pi_n = \sum_{n}\rho^V_n = \infty,\qquad 
    \sum_{n}{\rho^\Pi_n}^2 < \infty,\ \sum_{n}{\rho^V_n}^2 < \infty,
\end{equation}
under these conditions, the critic converges to $V^{\pi_\psi}$, while the actor asymptotically ascends the performance objective $J(\psi)$.

\section{Actor-Critic Algorithms for Infinite Horizon MFG and MFC Problems}
\label{sec:algos}

\subsection{MFG Algorithm}

The algorithm employs two neural networks: one representing the actor and one representing the critic. It follows the classical actor-critic RL framework, with the key difference that it incorporates the population distribution, which is required by the environment to compute rewards and state transitions. This distribution must be estimated and updated along the trajectory.

We adopt a simplified version of the algorithm from \cite{angiuli2023deep}, where the population distribution is updated directly (line \ref{mu} in Algorithm \ref{algoMFG}) rather than via the score function. Here, $\mu$ denotes an empirical distribution that is updated and stored as a weighted list of the trajectory. This simplification is motivated by two factors: (i) it remains tractable in the MFC setting, whereas using score functions may become intractable, and (ii) in many applications, the cost function $f$ and transition kernel $p$ depend smoothly on the population distribution—often through its mean—which can be updated directly. 

Note that this empirical distribution is necessary because learning occurs along a single trajectory of the state process; it is not required in the idealized setting described in Section~\ref{sec:idealized}. 

\begin{algorithm}[H]
\caption{Infinite Horizon  Mean Field Game Actor-Critic (IH-MFG-AC)} \label{algoMFG}
\begin{algorithmic}[1]
    \REQUIRE 
    Initial distribution $\mu$; number of time steps $N$; time-dependent neural network learning rates for actor $\rho_n^\Pi$, critic $\rho_n^V$; learning rates $\rho_n$ for the population distribution; discount factor $\gamma<1$
    \STATE \textbf{Initialize neural networks:} 
    Actor $\Pi_{\psi_0}$ : $\mathbbm{R}^d\to\mathcal{P}(\mathbbm{R}^k)$, Critic $V_{\theta_0}$ : $\mathbbm{R}^d\to\mathbbm{R}$
    \STATE \textbf{Sample initial state:} $X_0 \sim \mu$ and save $\mu_0$ as $\delta_{X_0}$
    \FOR{$n=0,1,2,\dots, N-1$}
        \STATE \textbf{Sample action:} $A_n \sim \Pi_{\psi_n}(\cdot \mid X_n)$
        \STATE \textbf{Observe reward from the environment:} $r_{n+1} = -f(X_n, \mu_n, A_n)$
        \STATE \textbf{Observe next state from the environment:} $X_{n+1} \sim p(X_n, \cdot, \mu_n, A_n)$
        \STATE \label{mu} \textbf{Update population measure:} 
        $\mu_{n+1} = \mu_n - \rho_n (\mu_n - \delta_{X_{n+1}})$ 
        (store empirical distribution as a weighted list)
        \STATE \textbf{Compute TD target (frozen):} $y_{n+1} = r_{n+1} + \gamma V_{\theta_n}(X_{n+1})$
        \STATE \textbf{Compute TD error:} $\delta_{\theta_n} = y_{n+1} - V_{\theta_n}(X_n)$
        \STATE \textbf{Compute critic loss:} $L_V^{(n)}(\theta_n) = \delta_{\theta_n}^2$
        \STATE \textbf{Update critic with SGD:} $\theta_{n+1} = \theta_n - \rho_n^V \nabla_\theta L_V^{(n)}(\theta_n)$
        \STATE \textbf{Compute actor loss:} $L_\Pi^{(n)}(\psi_n) = -\delta_{\theta_n} \log \Pi_{\psi_n}(A_n \mid X_n)$ 
        (uses TD error as advantage estimate)
        \STATE \textbf{Update actor with SGD:} $\psi_{n+1} = \psi_n - \rho_n^\Pi \nabla_\psi L_\Pi^{(n)}(\psi_n)$
    \ENDFOR
    \STATE \textbf{Return} the final actor network and population distribution $(\Pi_{\psi_N}, \mu_N)$
\end{algorithmic}
\end{algorithm}

\subsection{MFC Algorithm}

The Bellman equation for MFC involves finding an optimal policy over an infinite space. A common approach to solving it is to discretize the population distribution so that the problem reduces to finding an optimal policy over a finite space. Inspired by our previous work~\cite{Andrea23} in the finite-space setting, where we introduced distributions depending on $(x,a)$, we extend this idea to the present continuous-space setting by discretizing both the state and action spaces into bins.  

Specifically, we construct a partition $P_{\mathcal{X}} := \{\mathcal{X}_1,\dots,\mathcal{X}_m\}$ of the state space and a partition $P_{\mathcal{A}} := \{\mathcal{A}_1,\dots,\mathcal{A}_l\}$ of the action space, where $\mathcal{X}_j$ and $\mathcal{A}_k$ are intervals. We then define the bins
\[
    B_{jk} = \mathcal{X}_j \times \mathcal{A}_k, \quad j \in \{1,2,\dots,m\}, \quad k \in \{1,2,\dots,l\}.
\]
The set of bins forms a partition of $\mathcal{X}\times\mathcal{A}$. Introducing a new index $i := (j-1)l + k$, we denote $B^i = B_{jk}$. This indexing is one-to-one, and we define the projection $\proj_{\mathcal{X}}(B^i) = \mathcal{X}_{\lceil i/l \rceil}$.  

For each bin indexed by $i \in \{1, \dots, ml\}$, we define a distribution $\mu^i$. Alongside, for each bin we generate a path; when the path enters the corresponding bin, we select the midpoint action of that bin instead of sampling from the general policy. We denote this midpoint action by $a^i := \midd(\proj_{\mathcal{A}}(B^i))$. Finally, we create an individual path with its own distribution $\mu$, which does not influence the critic network but is used to train the actor network.  

As in the MFG case, the population distributions $\mu^i$ and $\mu$ are empirical distributions, updated and stored as weighted lists of the corresponding trajectories in the state space.  

\begin{algorithm}[H]
\caption{Infinite Horizon  Mean Field Control Actor-Critic (IH-MFC-AC)} 
\label{algo:MFC}
\begin{algorithmic}[1]
    \REQUIRE 
    Initial distributions $\{\mu^i\}$ and $\mu$; number of time steps $N$; time-dependent learning rates for actor $\rho_n^\Pi$, critic $\rho_n^V$, and population distribution $\rho_n$; discount factor $\gamma < 1$.
    
    \STATE \textbf{Initialize neural networks:} 
    Actor $\Pi_{\psi_0}$ : $\mathbbm{R}^d \to \mathcal{P}(\mathbbm{R}^k)$; 
    Critics $V_{\theta^i_0}$ : $\mathbbm{R}^d \to \mathbbm{R}$, $i \in \{1,\dots,ml\}$.
    
    \STATE \textbf{Sample initial states:} $X^i_0 \sim \mu^i$, $X_0 \sim \mu$, and set $\mu^i_0 = \delta_{X^i_0}$, $\mu_0 = \delta_{X_0}$.
    
    \FOR{$n = 0, 1, 2, \dots, N-1$}
        \IF{$X^i_n \in \proj_{\mathcal{X}} B^i$}
            \STATE \textbf{Choose action:} $A^i_n = a^i$
            \STATE \textbf{Observe reward:} $r^{i}_{n+1} = -f(X^i_n, \mu^i_n, a^i)$
            \STATE \textbf{Observe next state:} $X^i_{n+1} \sim p(X^i_n, \cdot, \mu^i_n, a^i)$
            \STATE \textbf{Find minimum value at $X^{i}_{n+1}$:} 
            $V_n(X^{i}_{n+1}) = \min_{j} V_{\theta^j_n}(X^{i}_{n+1})$, 
            where $j$ satisfies $X^{i}_{n+1} \in \proj_{\mathcal{X}} B^j$
            \STATE \textbf{Compute TD target:} $y^{i}_{n+1} = r^{i}_{n+1} + \gamma V_n(X^{i}_{n+1})$
            \STATE \textbf{Compute TD error:} $\delta_{\theta^i_n} = y^{i}_{n+1} - V_{\theta^i_n}(X^i_n)$
            \STATE \textbf{Compute critic loss:} $L_V^{(n)}(\theta^i_n) = (\delta_{\theta^i_n})^2$
            \STATE \textbf{Update critic with SGD:} 
            $\theta^i_{n+1} = \theta^i_n - \rho^V_n \nabla_{\theta^i} L_V^{(n)}(\theta^i_n)$ 
            (freeze TD target and backpropagate only through $V_{\theta^i_n}$)
        \ELSE
            \STATE \textbf{Sample action:} $A^i_n \sim \Pi_{\psi_n}(\cdot \mid X^i_n)$
            \STATE \textbf{Observe next state:} $X^i_{n+1} \sim p(X^i_n, \cdot, \mu^i_n, A^i_n)$
        \ENDIF
        
        \STATE \textbf{Update measure for bin $i$:} 
        $\mu^i_{n+1} = \mu^i_n - \rho_n (\mu^i_n - \delta_{X^i_{n+1}})$ 
        (store as empirical weighted list)
        
        \STATE \textbf{Sample action for individual path:} $A_n \sim \Pi_{\psi_n}(\cdot \mid X_n)$
        \STATE \textbf{Observe reward:} $r_{n+1} = -f(X_n, \mu_n, A_n)$
        \STATE \textbf{Observe next state:} $X_{n+1} \sim p(X_n, \cdot, \mu_n, A_n)$
        \STATE \textbf{Find $V_n$:} 
        \(
        V_n(x) = \min_j V_{\theta^j_n}(x) \quad \text{for } j \text{ such that } x \in \proj_{\mathcal{X}} B^j
        \)
        \STATE \textbf{Compute TD target:} $y_{n+1} = r_{n+1} + \gamma V_n(X_{n+1})$
        \STATE \textbf{Compute TD error:} $\delta_{\theta_n} = y_{n+1} - V_n(X_n)$
        \STATE \textbf{Compute actor loss:} $L_\Pi^{(n)}(\psi_n) = -\delta_{\theta_n} \log \Pi_{\psi_n}(A_n \mid X_n)$
        \STATE \textbf{Update actor with SGD:} $\psi_{n+1} = \psi_n - \rho^\Pi_n \nabla_\psi L_\Pi^{(n)}(\psi_n)$
        \STATE \textbf{Update measure for individual path:} 
        $\mu_{n+1} = \mu_n - \rho_n (\mu_n - \delta_{X_{n+1}})$ 
    \ENDFOR
    
    \STATE \textbf{Return} $(\Pi_{\psi_N}, \mu_N)$
\end{algorithmic}
\end{algorithm}

\begin{remark}
In the numerical implementation described in Section~\ref{sec:numerics}, we slightly modify the algorithm as follows. 
Inspired by DeepSeek’s GRPO~\cite{shao2024deepseekmathpushinglimitsmathematical}, we use group actions (mini-batches) to reduce variance and improve performance. This approach requires either freezing the environment or using a generative model. We then sample multiple next states and use the mean loss to update the critic network,
\(\delta_{\theta^i_n} = \frac{1}{M} \sum_{m=0}^{M-1} \big(y^{i,m}_{n+1} - V_{\theta^i_n}(X^i_n)\big),\)
and similarly average across multiple actions to update the actor,
\(L_\Pi^{(n)}(\psi_n) = -\frac{1}{M} \sum_{m=0}^{M-1} \delta^m_{\theta_n} \log \Pi_{\psi_n}(A^m_n \mid X_n).\)
\end{remark}

\section{Convergence of Idealized MFG and MFC Algoritms}
\label{sec:problems}

\subsection{Idealized Three Time-scale 
Approach}\label{sec:idealized}
\subsubsection{Three time-scale approach for MFG}
To solve the MFG problem that we defined in Sections~\ref{sec:mfgdef}, let us first introduce an idealized deterministic three-timescale approach where we assume that we have access to the expectation with respect to the transition kernel for all states at the same time. Consequently, we consider the following iterative procedure, where all variables are updated at each iteration but at different rates, denoted by $\rho^\mu_n \in \mathbbm{R}_+$, $\rho^V_n \in \mathbbm{R}_+$, and $\rho_n^\Pi\in \mathbbm{R}_+$. Starting from an initial guess $(\mu_0,\psi_0,\theta_0)$,  define iteratively for $n = 0,1,\dots$:
\begin{align}
\label{eq:system-ideal-3}
    \begin{cases}
        \mu_{n+1} = \mu_n - \rho_n^{\mu}L_{\mathcal{P}}(\psi_n,\theta_n,\mu_n),\\
        \theta_{n+1} = \theta_n - \rho_n^V \nabla_\theta L_V(\psi_n,\theta_n,\mu_n),\\
        \psi_{n+1} = \psi_n - \rho^\Pi_n \nabla_\psi L_\Pi (\psi_n,\theta_n,\mu_n).
    \end{cases}
\end{align}
where $L_{\mathcal{P}}$, $L_V$, and $L_\Pi$ are the loss functions of the measure, the critic, and the actor networks defined as follows:
\begin{align}
    \begin{cases}
        L_{\mathcal{P}}(\psi,\theta,\mu)(x) =\mu(x)- \int_{x'}\mu(x')p(x|x',\Pi_{\psi}(a|x'),\mu),\\
        L_V(\psi,\theta,\mu)(x) = \left[f(x,\Pi_{\psi}(a|x),\mu) + e^{-\beta}\int_{x'}p(x'|x,\Pi_{\psi}(a|x),\mu)V_\theta(x')-V_\theta(x)\right]^2,\\
        L_\Pi(\psi,\theta,\mu)(x,a) = -\left[f(x,a,\mu) + e^{-\beta}\int_{x'}p(x'|x,a,\mu)V_\theta(x')-V_\theta(x)\right]\log\Pi_\psi(a|x).
    \end{cases}
\end{align}

For the actor function, we only consider actions such that $\Pi_\psi(a|x)>0$.
The corresponding learning rates satisfy the usual Robbins-Monro-type conditions which will be made precise in  Assumption~\ref{squaresumlearningrate_mfg}.

If $\rho_n^\mu<\rho_n^\Pi<\rho_n^V$ such that $\rho_n^\mu/\rho_n^\Pi \to 0$ and $\rho_n^\Pi/\rho_n^V \to 0$ as $n \to \infty$, the system~\eqref{eq:system-ideal-3} tracks the ODE system:
\begin{align}\label{MFG_system}
    \begin{cases}
        \dot{\mu} = -L_{\mathcal{P}}(\psi,\theta,\mu)\\
        \dot{\psi} = -\frac{1}{\epsilon}\nabla_\psi L_\Pi(\psi,\theta,\mu)\\
        \dot{\theta} = -\frac{1}{\epsilon\tilde{\epsilon}}\nabla_\theta L_V(\psi,\theta,\mu),
    \end{cases}
\end{align}
where $\rho_n^\mu/\rho_n^\Pi$ and $\rho_n^\Pi/\rho_n^V$ are thought of being order $\epsilon \ll 1$ and $\tilde{\epsilon}\ll 1$, respectively. The asymptotic analysis of this system relies on the existence of globally asymptotically stable equilibrium (GASE) for these  equations. This will be made precise in Section~\ref{sec:MFG}.

\subsubsection{Three time-scale approach for MFC}
In MFC, for each bin, we have a measure $\mu^i$ and a vector of parameters for the neural network of state value function $\theta^i$. For the individual path, we create only a measure denoted by $\mu$. The whole vector of $\theta^i$ is then denoted by $\boldsymbol{\theta}$. Finally, we create a control parameter $\psi$. For an admissible policy $\pi$, we define the MKV dynamics by $p(x'|x,a,\mu^\pi)$ so that $\mu^\pi$ is the limiting distribution of the associated process $(X_n^\pi)_n$. We define the policy $\Pi_\psi^i$ for each bin by
\begin{align}\label{alphacontrol}
   \Pi_\psi^i (x') := \begin{cases}
        \delta_{a^i}\hspace{3mm} \text{if}\hspace{5mm} x'\in \proj_{\mathcal{X}}B^i, \\
        \Pi_\psi(x')\hspace{2mm} \text{for}\hspace{3mm} x'\notin \proj_{\mathcal{X}}B^i.
    \end{cases} 
\end{align}

Then the idealized three-timescale approach defined for MFC is as follows, which slightly differ from the corresponding approach defined for MFG:
\begin{align}
    \begin{cases}
        \mu^i_{n+1} = \mu^i_n - \rho_n^{\mu}L^i_{\mathcal{P}}(\psi_n,\boldsymbol{\theta_n},\mu^i_n),\\
        \mu_{n+1} = \mu_n - \rho_n^{\mu}L_{\mathcal{P}}(\psi_n,\boldsymbol{\theta_n},\mu_n),\\
        \theta^i_{n+1} = \theta^i_n - \rho_n^V \nabla_{\theta^i} L^i_V(\psi_n,\boldsymbol{\theta_n},\mu^i_n),\\
        \psi_{n+1} = \psi_n - \rho^\Pi_n \nabla_\psi L_\Pi (\psi_n,\boldsymbol{\theta_n},\mu_n),
    \end{cases}
\end{align}
where $L_{\mathcal{P}}$, $L_V$, and $L_\Pi$ are the loss function of the measure, the critic, and the actor networks as following
\begin{align}\label{mfceq}
    \begin{cases}
        L^i_{\mathcal{P}}(\psi,\boldsymbol{\theta},\mu)(y) =\mu(y)-\int\mu(x)p(y|x,\Pi^i_\psi(a|x),\mu),\\
         L_{\mathcal{P}}(\psi,\boldsymbol{\theta},\mu)(y) = \mu(y)-\int_{x}\mu(x)p(y|x,\Pi_{\psi}(a|x),\mu),\\
        L^i_V(\psi,\boldsymbol{\theta},\mu)(x) = \left[f(x,a^i,\mu) + e^{-\beta}\int_{x'}p(x'|x,a^i,\mu)V_{\boldsymbol{\theta}}(x')-V_{\theta^i}(x)\right]^2 \quad \text{where}\quad x\in \proj_{\mathcal{X}}B^i,\\
        L_\Pi(\psi,\boldsymbol{\theta},\mu)(x,a) = -\left[f(x,a,\mu) + e^{-\beta}\int_{x'}p(x'|x,a,\mu)V_{\boldsymbol{\theta}}(x')-V_{\boldsymbol{\theta}}(x)\right]\log\Pi_\psi(a|x),
    \end{cases}
\end{align}
where the value function $V_{\boldsymbol{\theta}}$ is defined as follows
\begin{align}
    V_{\boldsymbol{\theta}}(x) = \min_{j} V_{\theta^j}(x) \text{ for } j \text{ such that } x\in \proj_{\mathcal{X}}B^j.
\end{align}

 If $\rho_n^\Pi<\rho_n^V<\rho_n^\mu$ such that $\rho_n^\Pi/\rho_n^V \to 0$ and $\rho_n^V/\rho_n^\mu \to 0$ as $n \to \infty$, the system~\eqref{eq:system-ideal-3} tracks the ODE system:
\begin{align}\label{MFC_system}
    \begin{cases}
        \dot{\psi} = -\nabla_\psi L_\Pi(\psi,\boldsymbol{\theta},\mu)\\
        \dot{\theta^i} = -\frac{1}{\epsilon}\nabla_\theta L^i_V(\psi,\boldsymbol{\theta},\mu^i)\\
        \dot{\mu^i} = -\frac{1}{\epsilon\tilde{\epsilon}} L^i_{\mathcal{P}}(\psi,\boldsymbol{\theta},\mu^i)\\
        \dot{\mu} = -\frac{1}{\epsilon\tilde{\epsilon}} L_{\mathcal{P}}(\psi,\boldsymbol{\theta},\mu)\\
    \end{cases}
\end{align}
where $\rho_n^\Pi/\rho_n^V$ and $\rho_n^V/\rho_n^\mu $ are thought of being order $\epsilon \ll 1$ and $\tilde{\epsilon}\ll 1$, respectively. The asymptotic analysis of this system relies on the existence of globally asymptotically stable equilibrium (GASE) for the four equations. This will be made precise in Section \ref{sec:MFC}.

\subsubsection{Assumptions}
In order to establish the convergence result, we make the following assumptions. For the value function $V_\theta$, we consider the following linear parametrization: for all $x\in\mathcal{X}$,
$$
V_{\theta}(x) =\sum_{k=1}^m \phi_k(x)^\top \theta,
$$
where $\theta\in \mathbbm{R}^n$ and $\phi_k:\mathcal{X}\to\mathbbm{R}^n$ is a basis function generating $n$-dimensional features for any state $x$ which satisfy the following property. 
\begin{assumption}\label{ACM3}
The basis functions $\{\phi_k=(\phi_k(x),\hspace{1mm} x\in\mathcal{X}), k=1,2,3,...,m\}$  are linearly independent and $\Phi =[\phi_k] $ has full rank. Also, for every $\theta\in\mathbbm{R}^n$, $\Phi^\top\theta \neq e$, where $e$ is the m-dimensional vector with all entries equal to one. 
\end{assumption}

We need an assumption ensuring that
    under any admissible policy $\pi\in {\cal A}$, the Markov chain resulting from the MDP $\{X^\pi_n,n=0,1,2,...\}$ is ergodic and therefore has a unique  limiting distribution.

\begin{assumption}\label{ACM1}
(Dobrushin condition from \cite{doi:10.1137/S0040585X97986825})
There exists $\alpha>0$ such that for all $x,y \in\mathcal{X}$ and $\pi\in\mathcal{A}$,
\[
\sup_{\mu,\nu}\|P(\cdot|x,\pi(x),\mu) - P(\cdot|y,\pi(y),\nu)\|_1\leq 2(1-\alpha).
\]

    There exists $\lambda \in [0,\alpha)$, such that for all $x\in\mathcal{X}$ and $\mu,\nu \in \mathcal{P}(\mathcal{X})$ and admissible policy $\pi\in\mathcal{A}$ one has
    \[
    d_{TV}(P(\cdot|x,\pi(x),\mu),P(\cdot|x,\pi(x),\nu)\leq \lambda d_{TV}(\mu,\nu)
    \]
\end{assumption}

Then by Theorem 2.2 in \cite{doi:10.1137/S0040585X97986825}, Assumption \ref{ACM1} is an optimal condition for the existence of a unique invariant distribution.

\begin{assumption}\label{ACM2}
    For any pair of state/action, $\pi_\psi(a|x)$ is continuously differentiable with respect to the parameter $\psi$.
\end{assumption}

\subsection{Convergence result of MFG algorithm}\label{sec:MFG}
In this section, we will establish the convergence of the idealized three-timescale approach defined in \eqref{MFG_system}. To obtain the MFG solution, the measure is updated slower than the actor-critic system.

\begin{assumption}\label{squaresumlearningrate_mfg} The learning rate $\rho^Q_n$ and $\rho^{\mu}_n$ are sequences of positive real numbers satisfying
\begin{align*}
    \sum_n \rho^\mu_n = \sum_n \rho^{V}_n = \sum_n \rho^{\Pi}_n = \infty, \quad \sum_n |\rho^\mu_n|^2 + |\rho^{V}_n|^2 + |\rho^{\Pi}_n|^2 <\infty,
    \quad \rho^\mu/\rho^\Pi_n \xrightarrow[n\to+\infty]{} 0,
    \quad \rho^\Pi_n/\rho^V_n \xrightarrow[n\to+\infty]{} 0.
\end{align*}
\end{assumption}

\begin{assumption}\label{MFG_ACGASE}
    For fixed $\mu$, the ODEs of actor-critic system (the second and third ODEs in~\eqref{MFG_system}) have  unique globally asymptotically stable equilibrium (GASE), denoted by $(\theta^*_\mu,\psi^*_\mu)$, so that $\nabla_\theta L_V(\psi^*_\mu,\theta^*_\mu,\mu)=0$ and $\nabla_\psi L_\Pi(\psi^*_\mu,\theta^*_\mu,\mu)=0$.
\end{assumption}
Consequently, we obtain the following result. 
\begin{proposition}\label{MFG_ACGASEPRO}
    Under Assumptions \ref{ACM1}, \ref{ACM2}, \ref{ACM3}, \ref{squaresumlearningrate_mfg} and \ref{MFG_ACGASE}, for fixed $\mu$, as $n\to\infty$, $\theta_n \to \theta^*_\mu$ and $\psi_n \to \psi^*_\mu$ where $(\theta^*_\mu,\psi^*_\mu)$ is the GASE mentioned in Assumption \ref{MFG_ACGASE}.
\end{proposition}
\begin{proof}
    Under Assumption \ref{ACM1}, \ref{ACM2}, \ref{ACM3}, \ref{squaresumlearningrate_mfg} and \ref{MFG_ACGASE}, we can apply Theorem 1 in \cite{NIPS2007_6883966f} and Theorem~3 in \cite{NIPS1999_6449f44a}. Consequently this result guarantees the convergence in the statement.
\end{proof}
Next, we consider the convergence of $\mu$ and then obtain the unique limiting distribution. 
\begin{assumption}\label{secondGASE_mfg}
    The first ODE in \eqref{MFG_system} has a unique GASE, denoted by $\mu^*$ such that $ L_{\mathcal{P}}(\psi^*_{\mu^*},\theta^*_{\mu^*},\mu^*) = 0$.
\end{assumption}
We now make precise our earlier statement that  $f$ and $p$ depend smoothly on $\mu$.
\begin{lemma}\label{lipmu} Assuming that $f$ and $p$ are Lipschitz  with respect to $\mu$, then 
    $L_V$, $L_\Pi$, $L_{\mathcal{P}}$ are also Lipschitz respect to $\mu$.
\end{lemma}

\begin{proposition}\label{mfg_converge}
    Under Assumptions \ref{ACM1}, \ref{ACM2}, \ref{ACM3}, \ref{squaresumlearningrate_mfg}, \ref{secondGASE_mfg}, $\mu_n\to\mu^*$, $\theta_n \to \theta^*_{\mu^*}$ and $\psi_n \to \psi^*_{\mu^*}$, as $n\to\infty$.
\end{proposition}
\begin{proof}
    With Assumptions \ref{ACM1}, \ref{ACM2}, \ref{ACM3}, \ref{squaresumlearningrate_mfg}, \ref{secondGASE_mfg},  and the results of Propositions \ref{MFG_ACGASEPRO} and \ref{lipmu}, the assumptions of \cite[Theorem~1.1]{Borkar97} are satisfied, and therefore  guarantees the convergence in the statement.
\end{proof}
\begin{theorem}\label{mainThMFG}
    $(\mu^*,\psi^*_{\mu^*},\theta^*_{\mu^*})$ forms a solution of the MFG problem.
\end{theorem}
\begin{proof}
        From $(\mu^*,\psi^*_{\mu^*},\theta^*_{\mu^*})$ we can uniquely obtain the corresponding $(\mu^*,\pi^*,V^*)$ such that they satisfy the following Bellman equation. Given the optimal $\pi^* = \Pi_{\psi^*_{\mu^*}}$
        \[
            V_{\theta^*_{\mu^*}}(x) = \mathbbm{E}_{\pi^*}\left[r_{n+1}+\gamma V_{\theta^*_{\mu^*}}(X_{n+1})|X_{n}=x\right] 
        .\]
        Consequently they form an MFG solution.
\end{proof}

\subsection{Convergence result of MFC Algorithm}\label{sec:MFC}
In this section, we establish the convergence of the idealized three-timescale approach defined in \eqref{MFC_system}. Under MFC, the measures are updated faster than the actor-critic system.

\begin{assumption}\label{squaresumlearningrate_mfc} The learning rate $\rho^Q_n$ and $\rho^{\mu}_n$ are sequences of positive real numbers satisfying
\begin{align*}
    \sum_n \rho^\mu_n = \sum_n \rho^{V}_n = \sum_n \rho^{\Pi}_n = \infty, \quad \sum_n |\rho^\mu_n|^2 + |\rho^{V}_n|^2 + |\rho^{\Pi}_n|^2 <\infty,
    \quad \rho^\Pi_n/\rho^V_n \xrightarrow[n\to+\infty]{} 0,
    \quad \rho^V_n/\rho^\mu_n \xrightarrow[n\to+\infty]{} 0.
\end{align*}
\end{assumption}
\begin{proposition}\label{mfclip}
    $L^i_{\mathcal{P}}$, $L_{\mathcal{P}}$, $L_V$ and $L_\Pi$ are Lipschitz with respect to $\mu$, $\theta$ and $\psi$.
\end{proposition}
\begin{proof}
    It directly follows from Proposition \ref{lipmu} and the linear parametrization of the neural networks.
\end{proof}
\begin{proposition}\label{gase_mu}
        For a fixed $\psi$ and $\theta$, the last two ODEs in \eqref{MFC_system} have unique GASEs, denoted respectively by $\mu^{i*}_\psi$ and $\mu^*_\psi$ such that $  L^i_{\mathcal{P}}(\psi,\theta,\mu^{i*}_\psi)=0$ and $L_{\mathcal{P}}(\psi,\theta,\mu^*_\psi)=0$.
\end{proposition}
\begin{proof}
We denote by $\mathrm{P}^{{\pi},{\mu}}$ the transition kernel according to the distribution ${\mu}$ on $\mathcal{X}$ and the policy ${{\pi}}\in \mathcal{A}$,   defined for any  distribution $\tilde\mu$ by:
\begin{equation}
\label{eq:def-Ptransitions}
    ( \tilde\mu\mathrm{P}^{\pi,\mu})(x) = \int_{x'\in \mathcal{X}} \tilde\mu(x') \int_{a\in \mathcal{A}} \pi(a|x') p(x|x',a,\mu), \qquad x \in \mathcal{X}.
\end{equation}
    We claim that $-L^i_{\mathcal{P}}(\psi,\boldsymbol{\theta},\mu) = \mu\mathrm{P}^{\Pi^i_\psi,\mu} - \mu $ and $-L_{\mathcal{P}}(\psi,\boldsymbol{\theta},\mu) = \mu\mathrm{P}^{\Pi_\psi,\mu} - \mu$. By Assumption \ref{ACM1}, one can easily check that $ \|\mu\mathrm{P}^{\Pi_\psi,\mu}-\nu\mathrm{P}^{\Pi_\psi,\nu}\|_1 \leq (1-\alpha+\lambda)\|\mu-\nu\|_1$, where $1-\alpha+\lambda < 1$ implies that $\mu \mapsto \mathrm{P}^{\Pi_\psi, \mu}\mu$ is a strict contraction and $\mu \mapsto \mathrm{P}^{\Pi^i_\psi, \mu}\mu$ is a strict contraction as well. As a result, by contraction mapping theorem \cite{SELL197342}, a unique GASE exists.
\end{proof}

Consequently, we promptly obtain the following result.
\begin{proposition}\label{MFC_MuGASE}
    Under Assumptions \ref{ACM1}, \ref{ACM2}, \ref{ACM3},  \ref{squaresumlearningrate_mfc} and \ref{MFC_MuGASE}, for fixed $\psi$, as $n\to\infty$, $\mu^i_n \to \mu^{i*}_\psi$ and $\mu_n \to \mu^*_\psi$ where $\mu^{i*}_\psi$ and $\mu^*_\psi$ are the GASE mentioned in Assumption \ref{MFC_MuGASE}.
\end{proposition}
\begin{proof}
     In Proposition \ref{gase_mu}, we proved the unique GASE exists and then by \cite[Theorem~3.1]{563625}, it is the limit of $\mu^i_t$ and $\mu_t$ denoted by ${\mu}^{*i}_\psi$ and $\mu^*_\psi$ for all bins and the individual path.
\end{proof}

Next, we consider the convergence of actor-critic network and then obtain the unique control.
\begin{assumption}\label{secondGASE_mfc}
    The first two ODE in \eqref{MFC_system} has a unique GASE, denoted by $\psi^*$ and $\boldsymbol{\theta}^*$ such that $\nabla_\theta L^i_V(\psi^*,\boldsymbol{\theta}^*,\mu^{i*}_{\psi^*})=0$ and $\nabla_\psi L_\Pi(\psi^*,\boldsymbol{\theta}^*,\mu^*_{\psi^*})=0$.
\end{assumption}
Consequently we obtain the main result.
\begin{theorem}\label{mfc_converge}
    Under Assumptions \ref{ACM1}, \ref{ACM2}, \ref{ACM3}, \ref{squaresumlearningrate_mfc}, \ref{secondGASE_mfc}, and result in Propositions \ref{mfclip} and \ref{MFC_MuGASE} we have  $\mu^i_n\to\mu^{i*}_{\psi^*}$, $\mu_n\to\mu^*_{\psi^*}$, $\boldsymbol{\theta}_n \to \boldsymbol{\theta}^*$, and $\psi_n \to \psi^*$, as $n\to\infty$, and $(\mu^*_{\psi^*}, \Pi_{\psi^*})$ forms an MFC solution.
\end{theorem}

\begin{proof}
    With Assumptions \ref{ACM1}, \ref{ACM2}, \ref{ACM3}, \ref{squaresumlearningrate_mfc}, \ref{secondGASE_mfc},  and the results of Propositions \ref{mfclip} and \ref{MFC_MuGASE}, we can apply Theorem~1 in \cite{NIPS2007_6883966f} and Theorem~3 in \cite{NIPS1999_6449f44a}. Hence, the assumptions of \cite[Theorem~1.1]{Borkar97} are satisfied and, therefore, this result guarantees the convergence in the statement.

Next, we want to characterize this limit and prove that it forms an MFC solution. We know the limiting point satisfies the following equations. To make the notation clearer, we use $V_{\theta^*}^i$ as a new notation for $V_{{\theta^*}^i}$ defined in \eqref{mfceq}.

\begin{align*}
    \begin{cases}
\mu^{i*}_{\psi^*}(y) = \int \mu^{i*}_{\psi^*}(x)p(y|x,\Pi^i_{\psi^*}(a|x),\mu^{i*}_{\psi^*}),\\
        V_{\theta^*}^i(x) = f(x,a^i,\mu^{i*}_{\psi^*}) + \gamma\int_{x'}p(x'|x,a^i,\mu^{i*}_{\psi^*})V_{\boldsymbol{\theta}}(x')\quad \text{where}\quad x\in \proj_{\mathcal{X}}B^i,\\
\mu^*_{\psi^*}(y) = \int\mu^*_{\psi^*}(x)p(y|x',\Pi_{\psi^*}(a|x),\mu^*_{\psi^*}),\\
V_{\boldsymbol{\theta}}(x) = f(x,\Pi_{\psi^*}(a|x),\mu^*_{\psi^*}) + \gamma\int_{x'}p(x'|x,\Pi_{\psi^*}(a|x),\mu^*_{\psi^*})V_{\boldsymbol{\theta}}(x').
    \end{cases}
\end{align*}
Next we will prove our system of value functions satisfy the Bellman equation established in \cite{doi:10.1287/opre.2022.2395}. Recall the integrated averaged reward function $\hat{r}$ is defined as follows,
\[
    \hat{r}(\mu,\pi) = \int_{s} \mu(ds)\int_a \pi(a|s)(da) r(s,a,\mu).
\]
In short, it can be rewritten as 
\[
    \hat{r}(\mu,\pi) = \int_{s} \mu(ds) r(s,\pi(a|s),\mu).
\]
Then, using this integrated averaged reward function, we can define the integrated Q function and obtain the Bellman equation through dynamic programming principle. As for any distribution $\mu$ and policy $\pi$,
\[
Q^{\pi^*}(\mu,\pi) = \hat{r}(\mu,\pi) + \gamma Q^{\pi^*}(\Phi(\mu,\pi),\pi^*)
\]
where $Q^{\pi^*} = \sup_{\pi'} Q^{\pi'}$. If we consider this process starts directly from its invariant distribution, we will get,
\[
Q^{\pi^*}(\mu^\pi,\pi) = \hat{r}(\mu^\pi,\pi) + \gamma Q^{\pi^*}(\mu^\pi,\pi^*).
\]

Under our scenario, we use cost function instead of reward, and we use the invariant distribution to calculate the cost.  In other word, we consider the following $\hat{r}$,
\[
    \hat{r}(\mu^\pi,\pi) = \int_{s} \mu^\pi(ds) f(s,\pi(a|s),\mu^\pi).
\]
Consequently if we let $\pi = \Pi^i_{\psi^*}$, we get the following result,
\[
    v(\mu^{\Pi^i_{\psi^*}}) = Q^{\Pi_{\psi^*}}(\mu^{\Pi^i_{\psi^*}},\Pi^i_{\psi^*}).
\]
Eventually, from our algorithm, if we consider not a fixed initial point but an initial point coming from its invariant distribution, we will get 
\[
Q^{\Pi_{\psi^*}}(\mu^{\Pi^i_{\psi^*}},\Pi^i_{\psi^*}) = \hat{r}(\mu^{\Pi^i_{\psi^*}},\Pi^i_{\psi^*}) + \gamma Q^{\Pi_{\psi^*}}(\mu^{\Pi^i_{\psi^*}},\Pi_{\psi^*})
\]
for all bins. As the number of bins increases to infinity, we can conclude that $\Pi_{\psi^*}$ is the best policy locally. By the assumption of uniqueness of the MFC solution, we claim that the policy $\Pi_{\psi^*}$ and the distribution $\mu^{\Pi_{\psi^*}}$ (which is the same as $\mu^*_{\psi^*}$) form an MFC solution.
\end{proof}

\section{Synchronous and Asynchronous Algorithms with Stochastic Approximation}\label{sec:fullalgo}

In the previous section, we introduced an idealized deterministic algorithm relying on expectations that we assumed we could compute perfectly. However, in practical situations, those expectations are unknown and instead we use samples. Then we can use the classical  theory of stochastic approximation to obtain convergence of algorithms with sample-based updates. Let us assume that for any $(x,a,\mu)$, the learner can know the value $f(x,a,\mu)$. Furthermore, the learner can sample a realization of the random variable
\[
    X'_{x,a,\mu}\sim p(\cdot|x,a,\mu).
\]
Then, the learner has access to realizations of the random variables $\widecheck{L}_{\mathcal{P}}(\psi,\theta,\mu)$, $\widecheck{L}^i_{\mathcal{P}}(\psi,\theta,\mu)(x')$ and $\widecheck{ L}_{\Pi}(\psi,\theta,\mu)(x,a)$ defined as follows:
\begin{align*}
    \begin{cases}
        \widecheck{L}_{\mathcal{P}}(\psi,\theta,\mu)(x') = \left(\mathbbm{\delta}_{\{X^{'}_{x,\Pi_\psi,\mu}=x'\}}-\mu(x')\right)_{x'\in\mathcal{X}}\\
        \widecheck{L}^i_{\mathcal{P}}(\psi,\theta,\mu)(x') = \left(\mathbbm{\delta}_{\{X^{'}_{x,\Pi^i_\psi,\mu}=x'\}}-\mu(x')\right)_{x'\in\mathcal{X}}\\
        \widecheck{ L}_{\Pi}(\psi,\theta,\mu)(x,a) = -\left[f(x,a,\mu) + \gamma V_{\boldsymbol{\theta}}(X'_{x,a,\mu})-V_{\boldsymbol{\theta}}(x)\right]\log\Pi_\psi(a|x).
    \end{cases}
\end{align*}
By \cite[Section 6]{Andrea23}, we can get that
\[
\mathbbm{E}[\widecheck{L}_{\mathcal{P}}(\psi,\theta,\mu)(x')] = L_{\mathcal{P}}(\psi,\theta,\mu)(x').
\]
Similarly we can check that
\begin{align*}
    \begin{cases}
        \mathbbm{E}[\widecheck{L}^i_{\mathcal{P}}(\psi,\theta,\mu)(x')] = L^i_{\mathcal{P}}(\psi,\theta,\mu)(x')
        \\
        \mathbbm{E}[\nabla_\psi\widecheck{ L}_{\Pi}(\psi,\theta,\mu)(x,a)] =\nabla_\psi L_{\Pi}(\psi,\theta,\mu)(x,a)
        \\
        \mathbbm{E}[\widecheck{ L}_{\Pi}(\psi,\theta,\mu)(x,a)] = L_{\Pi}(\psi,\theta,\mu)(x,a)
        \\
         \mathbbm{E}[\nabla_\theta\widecheck{ L}_{V}(\psi,\theta,\mu)(x)] =\nabla_\theta L_{V}(\psi,\theta,\mu)(x).
    \end{cases}
\end{align*}
Now we can define the synchronous algorithm for MFG with stochastic approximation:
\begin{align}
    \begin{cases}
        \mu_{n+1} = \mu_n - \rho_n^{\mu}\widecheck{L}_{\mathcal{P}}(\psi_n,\theta_n,\mu_n)+\rho_n^{\mu}M^{\mathcal{P}}_n,\\
        \theta_{n+1} = \theta_n - \rho_n^V \nabla_\theta L_V(\psi_n,\theta_n,\mu_n)+\rho_n^{V}M^{V}_n,\\
        \psi_{n+1} = \psi_n - \rho^\Pi_n \nabla_\psi \widecheck{ L}_{\Pi} (\psi_n,\theta_n,\mu_n) + \rho^\Pi_n M^{\Pi}_n ,
    \end{cases}
\end{align}
where $M^{\mathcal{P}}_n$, $M^{V}_n$ and $M^{\Pi}_n$ are martingales defined as follows,
\begin{align}
    \begin{cases}
        M^{\mathcal{P}}_n = \widecheck{L}_{\mathcal{P}}(\psi_n,\theta_n,\mu_n) - L_{\mathcal{P}}(\psi_n,\theta_n,\mu_n)\\
        M^{V}_n = \nabla_\theta\widecheck{ L}_{V}(\psi,\theta,\mu) - \nabla_\theta L_{V}(\psi,\theta,\mu)\\
        M^{\Pi}_n = \nabla_\psi \widecheck{ L}_{\Pi} (\psi_n,\theta_n,\mu_n) - \nabla_\psi L_{\Pi} (\psi_n,\theta_n,\mu_n).
    \end{cases}
\end{align}

In a similar way, the synchronous algorithm for MFC with stochastic approximation is defined as follows:
\begin{align}
    \begin{cases}
        \mu^i_{n+1} = \mu^i_n - \rho_n^{\mu}\widecheck{L}^i_{\mathcal{P}}(\psi_n,\boldsymbol{\theta_n},\mu^i_n)+\rho_n^{\mu}M^{\mathcal{\tilde P}}_n,\\
        \mu_{n+1} = \mu_n - \rho_n^{\mu}\widecheck{L}_{\mathcal{P}}(\psi_n,\boldsymbol{\theta_n},\mu_n)+\rho_n^{\mu}M^{\mathcal{P}}_n,\\
        \theta^i_{n+1} = \theta^i_n - \rho_n^V \nabla_{\theta^i} L^i_V(\psi_n,\boldsymbol{\theta_n},\mu^i_n),\\
        \psi_{n+1} = \psi_n - \rho^\Pi_n \nabla_\psi \widecheck{ L}_{\Pi} (\psi_n,\boldsymbol{\theta_n},\mu_n)+ \rho^\Pi_n M^{\Pi}_n.
    \end{cases}
\end{align}

Following the argument in \cite[Section 7]{Andrea23}, we can extend the algorithm under the asynchronous setting and their Theorem 7.6 (MFG) and Theorem 7.7 (MFC) guarantee the convergence results as in  Proposition~\ref{mfg_converge}  (MFG)  and Theorem~\ref{mfc_converge} (MFC) but with asynchronous feature described in Algorithms \ref{algoMFG} and \ref{algo:MFC}. We omit the details as they are very similar to those presented in \cite{Andrea23} for Q-learning and do not depend fundamentally on the differences between Q-learning and Actor-Critic learning.
 
\section{Numerical Results}\label{sec:Num}

We use linear quadratic example since there exists an analytic solution in continuous time. In the implementation, we use discrete time based on a suitable discretization of the problem.

\subsection{A Linear-Quadratic Benchmark}\label{sec:Num1d}
We begin by evaluating the performance of our algorithm on the benchmark example used in \cite{angiuli2023deep}. Note that our algorithms do not use score functions and are different for the MFG and MFC cases. The problem consists in optimizing 
\[
	\mathbbm{E}\left[\int_0^\infty e^{-\beta t}\left(\frac{1}{2}\alpha_t^2+c_1(X_t-c_2m)^2+c_3(X_t-c4)^2+c_5m^2\right)\mathrm{d}t\right]
\]
where $m=\int x\mu(\mathrm{d}x)$ and the state dynamics are: 
\[
	\mathrm{d}X_t = \alpha_t \mathrm{d}t+\sigma\mathrm{d}W_t,\quad t\in[0,\infty   ).
\] 

\subsubsection{Solution for Asymptotic Mean Field Game}

Traditional methods  for deriving the LQ problem solution begin with recovering the value function (details can be found for example in \cite{Andrea20})
\[
	v(x) := \inf_{\alpha \in \mathbbm{A}}\mathbbm{E}\left[\int_0^\infty e^{-\beta t}\left(\frac{1}{2}\alpha_t^2+c_1(X_t-c_2m)^2+c_3(X_t-c4)^2+c_5m^2\right)\mathrm{d}t|X_0=x\right],
\]
as the solution of a Hamilton-Jacobi-Bellman equation. In the MFG case, we denote the optimal value function as $\hat{v}$ and an expilict form is given by $\hat{v}(x) = \hat{\Gamma}_2x^2+\hat{\Gamma}_1x + \hat{\Gamma}_0$, where
\begin{align*}
\begin{cases}
	\hat\Gamma_2 = \frac{-\beta+\sqrt{\beta^2+8(c_1+c_3)}}{4}\\
	\hat{\Gamma}_1=-\frac{2\hat\Gamma_2c_3c_4}{\hat\Gamma_2(\beta+2\hat\Gamma_2)-c_1c_2}\\
	\hat\Gamma_0 = \frac{c_5\hat{m}^2+c_3c_4^2+c_1c_2^2\hat{m}^2+\sigma^2\hat\Gamma_2-\frac{1}{2}\hat{\Gamma}_1^2}{\beta}.
\end{cases}
\end{align*}
Then the optimal control for the MFG is 
\begin{align*}
    \hat{x}=-\hat{v}'(x) = -(2\hat{\Gamma}_2x+\hat{\Gamma}_1),
\end{align*}
and the limiting distribution of the corresponding process is 
\begin{align*}
    \hat\mu = \mathcal{N}\left(-\frac{\hat{\Gamma}_1}{2\hat{\Gamma}_2},\frac{\sigma^2}{4\hat{\Gamma}_2}\right).
\end{align*}
\subsubsection{Solution for Asymptotic Mean Field Control}
Similarly as above, we denote the MFC value function by $v^*$. It has the form $v^*(x)=\Gamma^*_2x^2+\Gamma^*_1x+\Gamma_0^*$ with 
\begin{align*}
    \begin{cases}
        \Gamma^*_2 = \frac{-\beta+\sqrt{\beta^2+8(c_1+c_3)}}{4}\\
{\Gamma}^*_1=-\frac{2\Gamma^*_2c_3c_4}{\Gamma^*_2(\beta+2\Gamma^*_2)+c_5-c_1c_2(2-c_2)}\\
\Gamma^*_0 = \frac{c_5{m^*}^2+c_3c_4^2+c_1c_2^2{m^*}^2+\sigma^2\Gamma^*_2-\frac{1}{2}{\Gamma^*_1}^2}{\beta}.
    \end{cases}
\end{align*}
The optimal control for the MFC is 
\begin{align*}
    \alpha^*(x) = -{v^*}' (x) = -(2\Gamma^*_2x+\Gamma_1^*),
\end{align*}
and the limiting distribution of the corresponding process is 
\begin{align*}
    \mu^*=\mathcal{N}\left(-\frac{\Gamma^*_1}{2\Gamma^*_2},\frac{\sigma^2}{4\Gamma^*_2}\right).
\end{align*}

\subsubsection{Numerical Details}\label{sec:numerics}

We proceed with the same choice of parameters as in \cite{angiuli2023deep} for comparison.

For our numerical experiment, we evaluated our algorithm using two distinct sets of values for the running cost coefficients $c_1$ through $c_5$ and the volatility $\sigma$, as detailed in Tables 2 and 3. In both cases, the discount factor is fixed at $\beta =1$, and we discretized continuous time with a step size of $\Delta t = 0.01$ ($\gamma = e^{-\beta\Delta t}$). Both the critic and score functions are implemented as feedforward neural networks featuring a hidden layer of 128 neurons with tanh activation. Similarly, the actor is designed as a feedforward neural network that outputs the mean and standard deviation of a normal distribution used for sampling actions. Its architecture consists of a shared hidden layer of 64 neurons with tanh activation, followed by two separate layers (each with 64 neurons) dedicated to computing the mean and standard deviation. To ensure that the standard deviation remains positive, its corresponding layer employs a softmax activation. As the actor is intended to converge to a deterministic or pure policy over time, we add a baseline value of $10^{-5}$ to the output layer to maintain a minimal level of exploration, effectively mirroring the entropy regularization approach described in \cite{wang2020reinforcement}. For further details, refer to Table~\ref{table:1} for the learning rates. 
We use  uniform distributions $\mu^i$ over sub-state space corresponding to each bin and take 42 bins in total, 6 under state space and 7 under action space.
The results for the following set of parameters are shown in Figures \ref{fig:1} and \ref{fig:2}.

\begin{table}[h!]
    \centering
    \begin{tabular}{ |p{1cm}|p{1.5cm}|p{1.5cm}|  }
\hline
& MFG &MFC \\
\hline
$\rho^\pi_0$ & $5\times 10^{-5}$ &$5\times 10^{-5}$ \\
$\rho^V_0$ & $10^{-4}$   & $10^{-4}$ \\
$\rho^\mu_0$ &$10^{-5}$ & $10^{-3}$ \\
\hline
\end{tabular}
    \caption{Choice of learning rate for MFG and MFC}
    \label{table:1}
\end{table}

\begin{table}[h!]
    \centering
    \begin{tabular}{ |p{0.7cm}|p{0.7cm}|p{0.7cm}|p{0.7cm}|p{0.7cm}|p{0.7cm}|  }
\hline
c1&c2&c3&c4&c5&$\sigma$ \\
\hline
$0.25$ & $1.5$&$0.5$&$0.6$&$1.0$&$0.3$ \\
\hline
\end{tabular}
    \caption{Running cost coefficients and volatility}
    \label{table:2}
\end{table}

\begin{figure}%
\center 
\subfloat[Result for MFG]
{\includegraphics[width=0.45\textwidth]{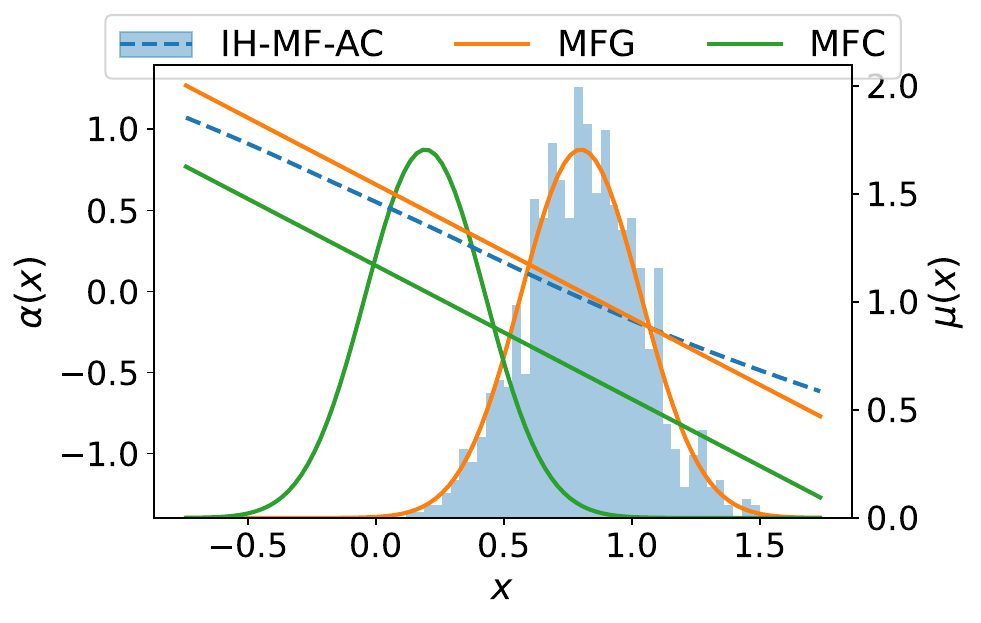}\label{fig:Ex_Im}}
\subfloat[Result for MFC]{\includegraphics[width=0.45\textwidth]{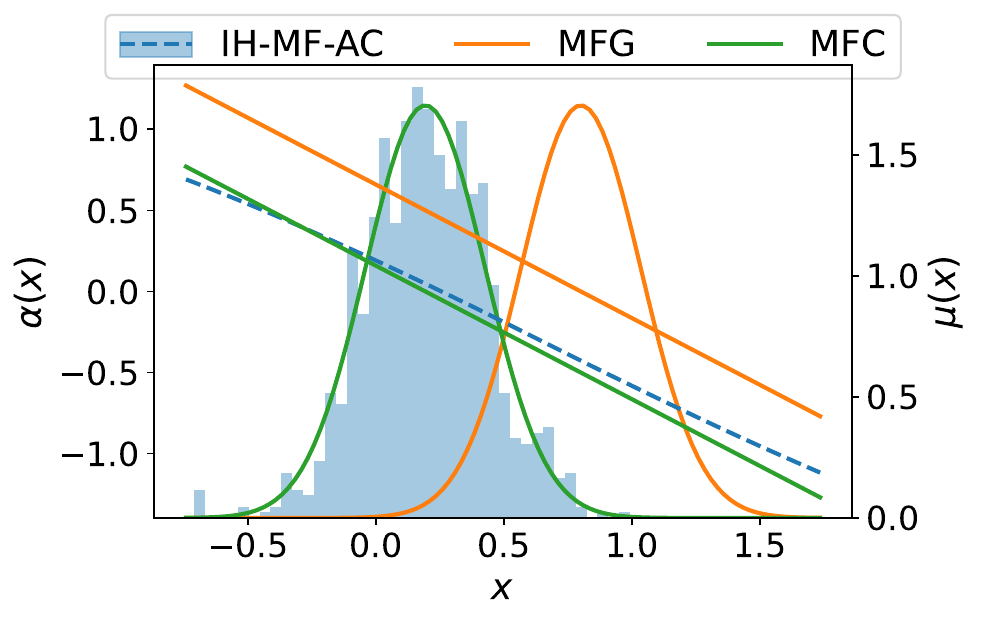}\label{fig:Ex_Im2}}
\caption{The histogram is the learned asymptotic distribution and the dashed line is the learned feedback control after $N=2\times10^5$ iterations. The green curves correspond to the optimal control and mean field distribution for MFC, while the orange curves are the equivalent for MFG. The bottom axis shows the state variable $x$, the left axis refers to the value of the control $\alpha(x)$, and the right axis represents the probability density for $\mu$.}
\label{fig:1}
\end{figure}

\begin{figure}%
\center 
\subfloat[Result for MFG]
{\includegraphics[width=0.45\textwidth]{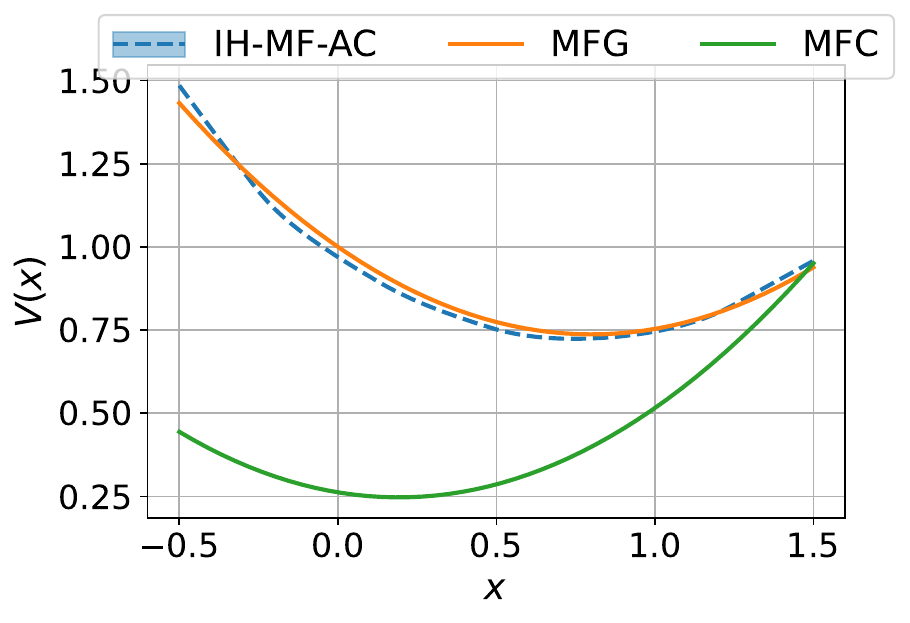}\label{fig:Ex_Im3}}
\subfloat[Result for MFC]{\includegraphics[width=0.45\textwidth]{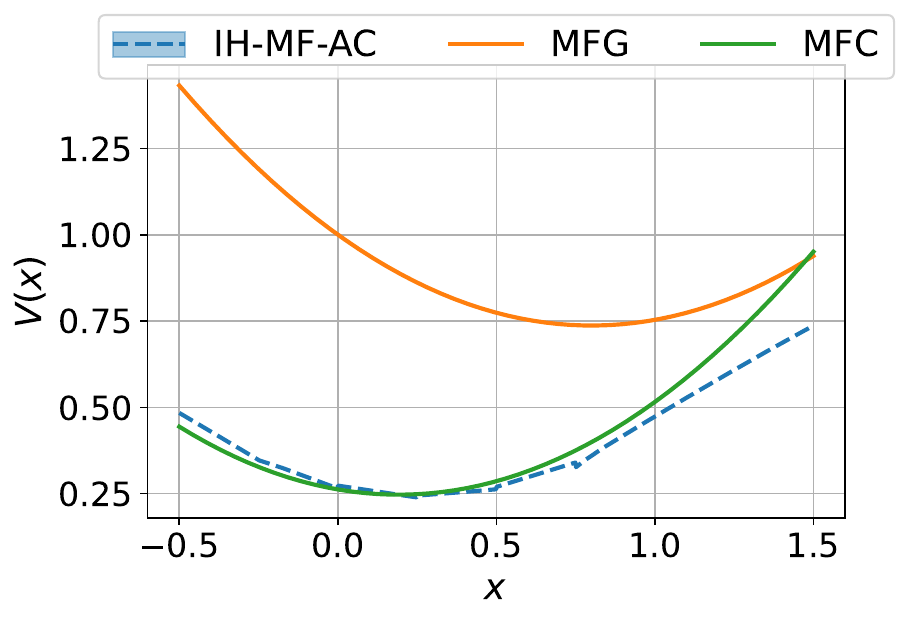}\label{fig:Ex_Im4}}
\caption{The orange and green curves are the optimal value functions for the MFG and MFC problem, respectively. The blue dashed line is the learned value function among all bins after $N=2\times10^5$ iterations. It is a poly-line for MFC, since we use bins}
\label{fig:2}
\end{figure}
In the plots in Figures \ref{fig:1} and \ref{fig:2}, we show the learned distributions, the learned value functions, and the learned policies for MFG and MFC. We also show the theoretical ones to identify the correctness of the learning results.

Furthermore, in order to quantify the error and show the robustness with respect to the learning rates, we consider the following error metrics for the MFG problem and their corresponding counterparts for the MFC problem (with hat replaced by star).

\begin{enumerate}
    \item The absolute error in the learned population mean:
    \[
        e_\mu := \|m_N-\hat{m}\|_2,
    \]
  \item The expected absolute error in the learned control:
  \[
    e_\alpha := \mathbbm{E}_{x\sim\hat\mu}[\|\alpha_{\psi_N}(x)-\hat{\alpha}(x)\|_2] \approx \frac{1}{l}\sum_{j=1}^l \|\alpha_{\psi_N}(x_j)-\hat{\alpha}(x_j)\|_2,
  \]
  \item The expected absolute error in the learned value function:
  \[
    e_{V} := \mathbbm{E}_{x\sim\hat\mu}[\|V_{\theta_N}(x)-\hat{v}(x)\|_2] \approx \frac{1}{l}\sum_{j=1}^l \|V_{\theta_N}(x_j)-\hat{v}(x_j)\|_2,
  \]
\end{enumerate}
where $x_j$ are i.i.d. with distribution $\hat\mu$ for $j = 1,2,...,l$.
    
In Table \ref{table:3}, we show the learning rates for the four considered  cases. The learning rate for the critic is  in the range $[10^{-3},10^{-4}]$ and the learning rate ratio between the critic and the actor is in the range $[2,5]$. 
We fix the population distribution learning rate in MFG as $10^{-6}$ and in MFC as $10^{-2}$.
Results for the three error metrics, $(e_\mu,e_\alpha,e_V)$ using $l=10000$, are shown in Table \ref{table:4} for MFG and in Table \ref{table:5} for MFC. We conclude that our algorithms are stable with respect to these hyper-parameters.

\begin{table}[h!]
    \centering
    \begin{tabular}{ |p{1cm}|p{1.5cm}|p{1.5cm}|p{1.5cm}|p{1.5cm}|  }
\hline
& Case 1 & Case 2 & Case 3 & Case 4 \\
\hline
$\rho^V$ & $10^{-3}$ & $10^{-4}$& $10^{-3}$ & $10^{-4}$ \\
$\rho^{\Pi}$&$5\times10^{-4}$ & $5\times10^{-5}$&$2\times10^{-4}$ & $2\times10^{-5}$\\
\hline
\end{tabular}
    \caption{Learning rate of different cases}
    \label{table:3}
\end{table}
\begin{table}[h!]
    \centering
    \begin{tabular}{ |p{1cm}|p{1.5cm}|p{1.5cm}|p{1.5cm}|p{1.5cm}|  }
\hline
& Case 1 & Case 2 & Case 3 & Case 4 \\
\hline
$e_{\mu}$ & 0.030 &0.028 &0.030 & 0.036 \\
$e_{\alpha}$ & 0.048 & 0.045 &0.046 & 0.050 \\
$e_{V}$ &0.054 & 0.046 & 0.054 &0.058\\
\hline
\end{tabular}
    \caption{MFG: Error for different choice of learning rate}
    \label{table:4}
\end{table}

\begin{table}[h!]
    \centering
    \begin{tabular}{ |p{1cm}|p{1.5cm}|p{1.5cm}|p{1.5cm}|p{1.5cm}|  }
\hline
& Case 1 & Case 2 & Case 3 & Case 4 \\
\hline
$e_{\mu}$ & 0.053 &0.045 &0.044 & 0.052 \\
$e_{\alpha}$ & 0.075 & 0.057 &0.069 & 0.072 \\
$e_{V}$ &0.060 & 0.049 & 0.065 & 0.062\\
\hline
\end{tabular}
    \caption{MFC: Error for different choice of learning rate}
    \label{table:5}
\end{table}
\subsection{A Two-Dimensional Linear-Quadratic Benchmark}
\label{sec:Num2d}
We also test our algorithm on a higher-dimensional generalization of the scalar LQ problem with the cost functional
\begin{align}
    \mathbbm{E}\left[\int_0^\infty e^{\beta t} \left(\frac{1}{2}\alpha_t^\top\alpha_t+(x-C_2m)^\top C_1(x-C_2m) + (x-c_4)^\top C_3(x-c_4) + m^\top C_5 m\right) \mathrm{d}t\right]
\end{align}
and state dynamics
\begin{align}
    \mathrm{d}X_t = \alpha_t\mathrm{d}t + \sigma \mathrm{d}W_t,\quad t\in [0,\infty).
\end{align}
where $C_1,C_3,C_5\in\mathbbm{R}^{d\times d}$ are positive-definite matrices, $c_4\in\mathbbm{R}^d$, $C_2\in\mathbbm{R}^{d\times d}$, $\sigma\in \mathbbm{R}^{d\times m}$, and $W_t$ is an m-dimensional Brownian motion. We test our algorithm on a benchmark problem for the case $d=m=2$.

\subsubsection{Numerical Details}
We randomly generated the following parameters which distinguish the MFG and MFC solutions and we use 36 bins for MFC, 9 for the state space and 4 for the action space. 
\setlength{\arrayrulewidth}{0.5mm}
\setlength{\tabcolsep}{8pt}
\renewcommand{\arraystretch}{1.5}
\begin{table}[h!]
    \centering
    \begin{tabular}{ |p{2.4cm}|p{2.4cm}|p{2.4cm}|p{1.3cm}|p{2.4cm}|p{1.7cm}|  }
\hline \centering
c1&\centering c2&\centering c3&\centering c4&\centering c5&$\quad\quad\sigma$ \\
\hline
$\begin{bmatrix}0.964&0.236\\0.236&0.076\end{bmatrix}$ &  $\begin{bmatrix}0.677&0.937\\0.937&1.357\end{bmatrix}$&
 $\begin{bmatrix}0.988&1.118\\1.118&1.483\end{bmatrix}$&
 $\begin{bmatrix}0.810\\0.872\end{bmatrix}$&
$\begin{bmatrix}1.511&0.072\\0.072&1.520\end{bmatrix}$&
$\begin{bmatrix}0.3&0.0\\0.0&0.3\end{bmatrix}$\\
\hline
\end{tabular}
    \caption{Running cost coefficients and volatility}
    \label{tabel:3}
\end{table}

The results for this set of parameters are shown in the following pictures.
\begin{figure}%
\center 
\subfloat[Result for MFG]
{\includegraphics[width=0.45\textwidth]{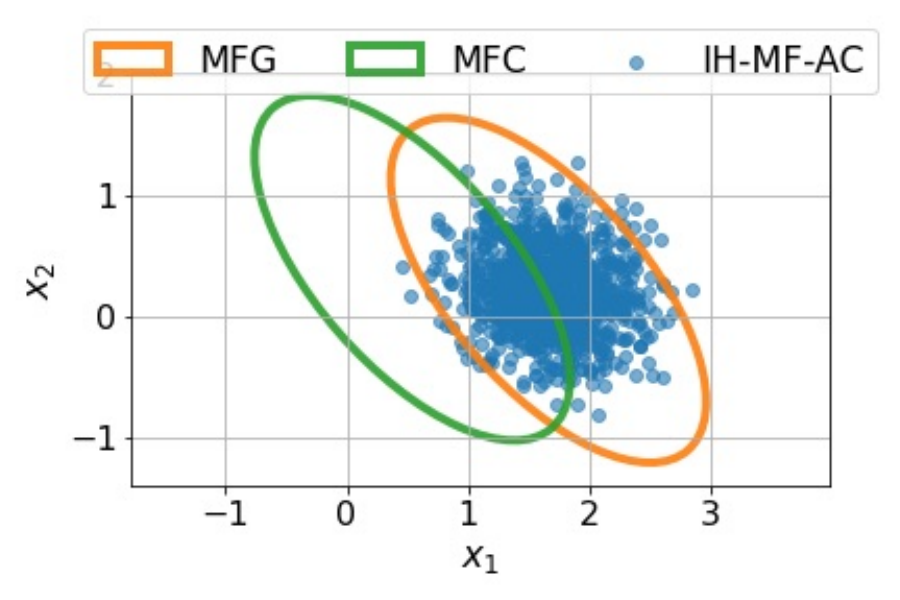}\label{fig:Ex_Im5}}
\subfloat[Result for MFC]{\includegraphics[width=0.45\textwidth]{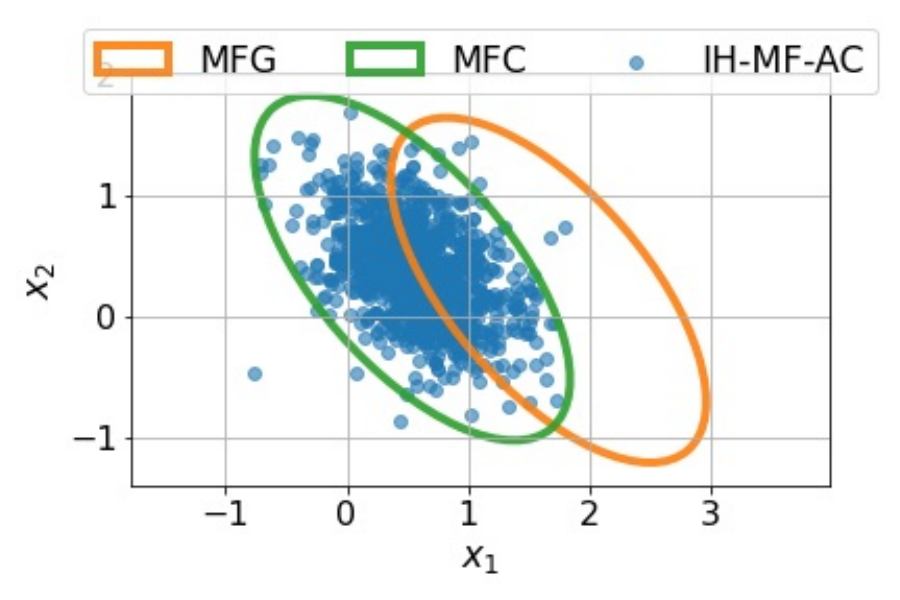}\label{fig:Ex_Im6}}
\caption{The scatter plot of points are the samples of the learned asymptotic distribution after $N=10^6$ iterations. The solid ellipses are the set of points with Mahalanobis distance $3$ from the optimal mean field distributions in the case of MFG (orange) and MFC (green).}
\label{fig:3}
\end{figure}

The plots in Figure \ref{fig:3}, shows the center and the set of points with Mahalanobis distance 3 from the optimal population distributions for MFG on the left  and MFC on the right. The dots are samples of the learned distributions; as desired, they are distributed within the theoretical ellipses. 

Note that the introduction of bins in the MFC algorithm significantly improves the identification of the limiting population distribution (Figure \ref{fig:Ex_Im6}) compared to the results in \cite{angiuli2023deep}.

\begin{figure}%
\center 
\subfloat[Result for MFG]
{\includegraphics[width=0.9\textwidth]{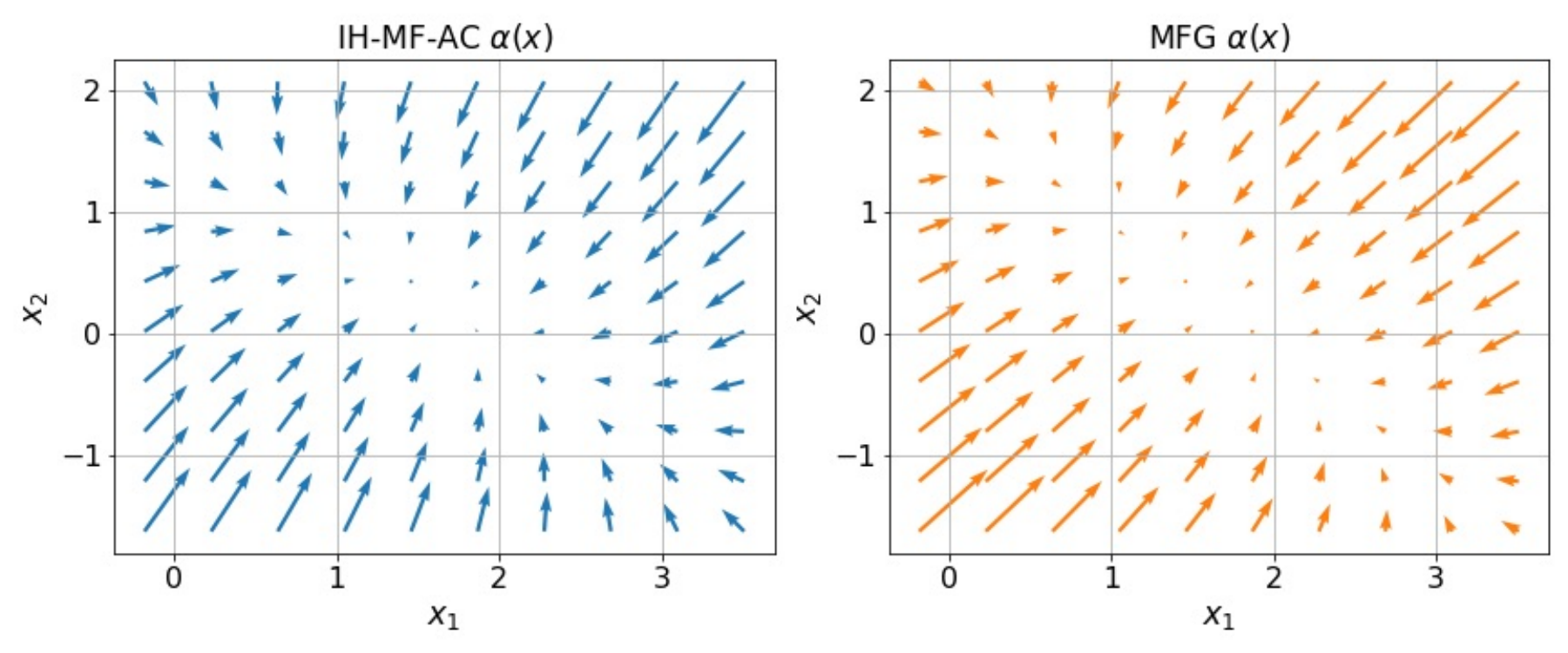}\label{fig:Ex_Im7}}\\
\subfloat[Result for MFC]{\includegraphics[width=0.90\textwidth]{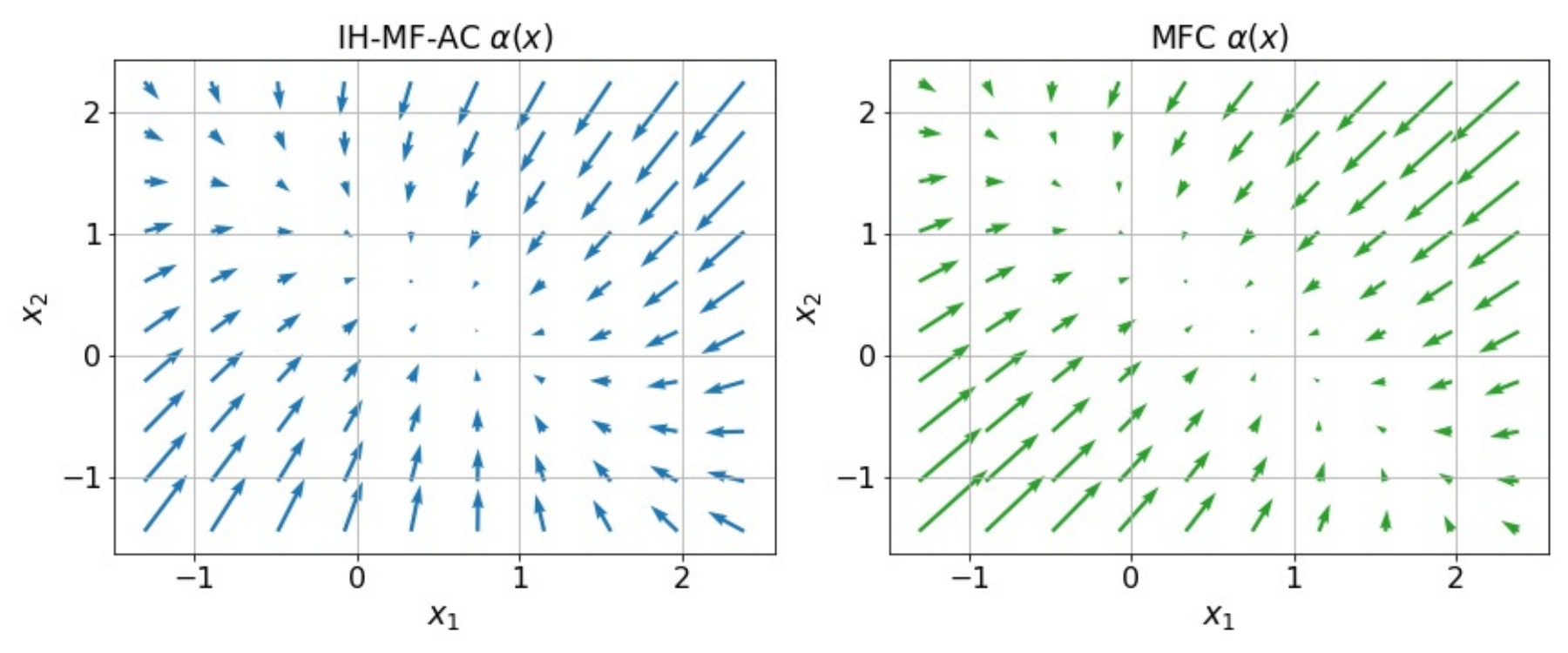}\label{fig:Ex_Im8}}
\caption{The orange and green vector fields are the optimal controls for the MFG and MFC problems, respectively. The blue vector field show the learned feedback controls after $N=10^6$ iterations.}
\label{fig:4}
\end{figure}

The plots in Figure \ref{fig:4} show the vector fields of learned (left) and theoretical (right) policies for MFG (top) and MFC (bottom). The learned policy within the learned population distribution from Figure \ref{fig:3} is visually quite close to the theoretical policy within the theoretical distribution also from Figure \ref{fig:3}. There are  inconsistencies outside the distribution support but this does not affect our mean absolute error. 

\begin{figure}%
\center 
\subfloat[Result for MFG]
{\includegraphics[width=0.9\textwidth]{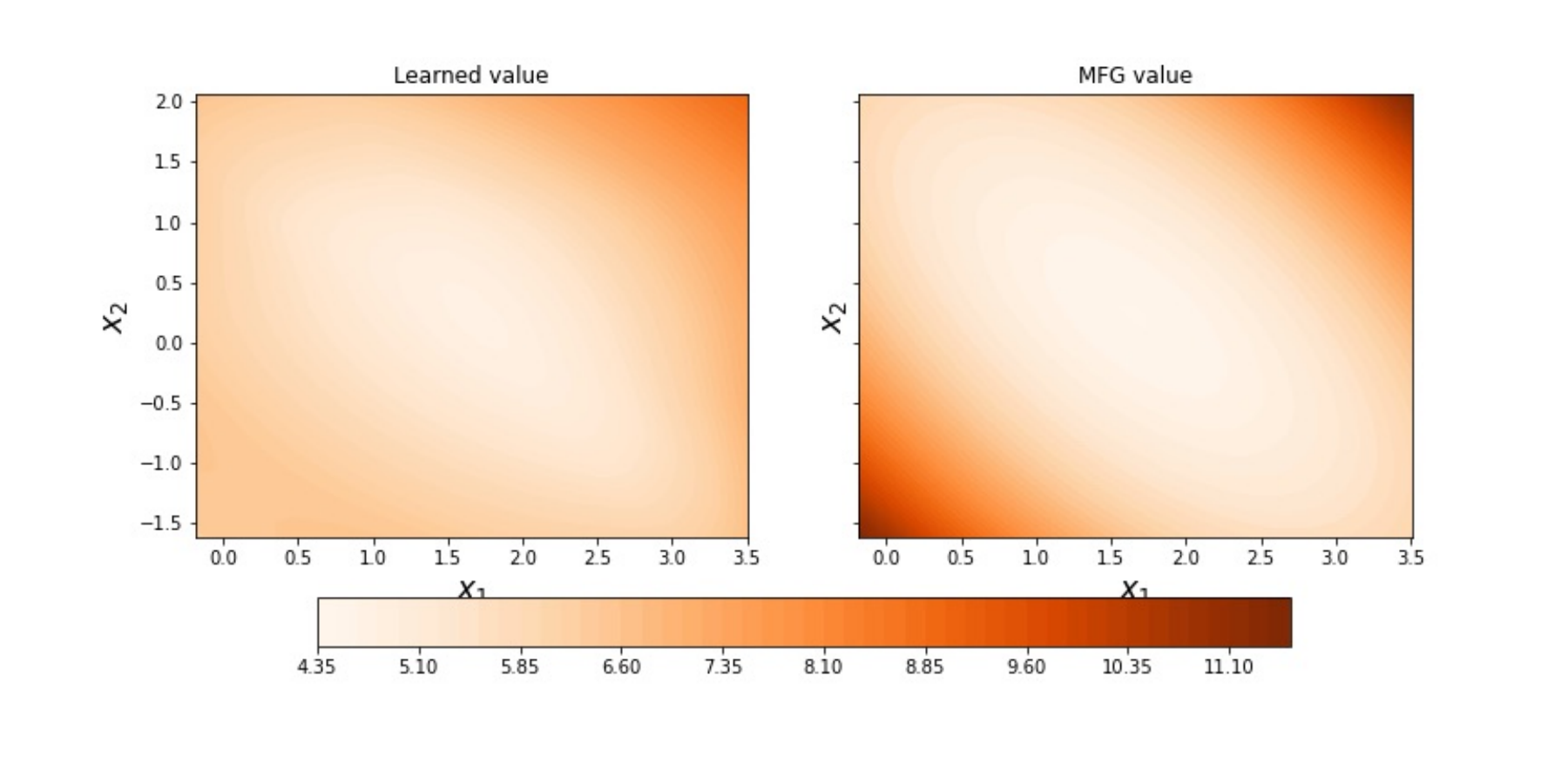}\label{fig:Ex_Im9}}\\
\subfloat[Result for MFC]{\includegraphics[width=0.90\textwidth]{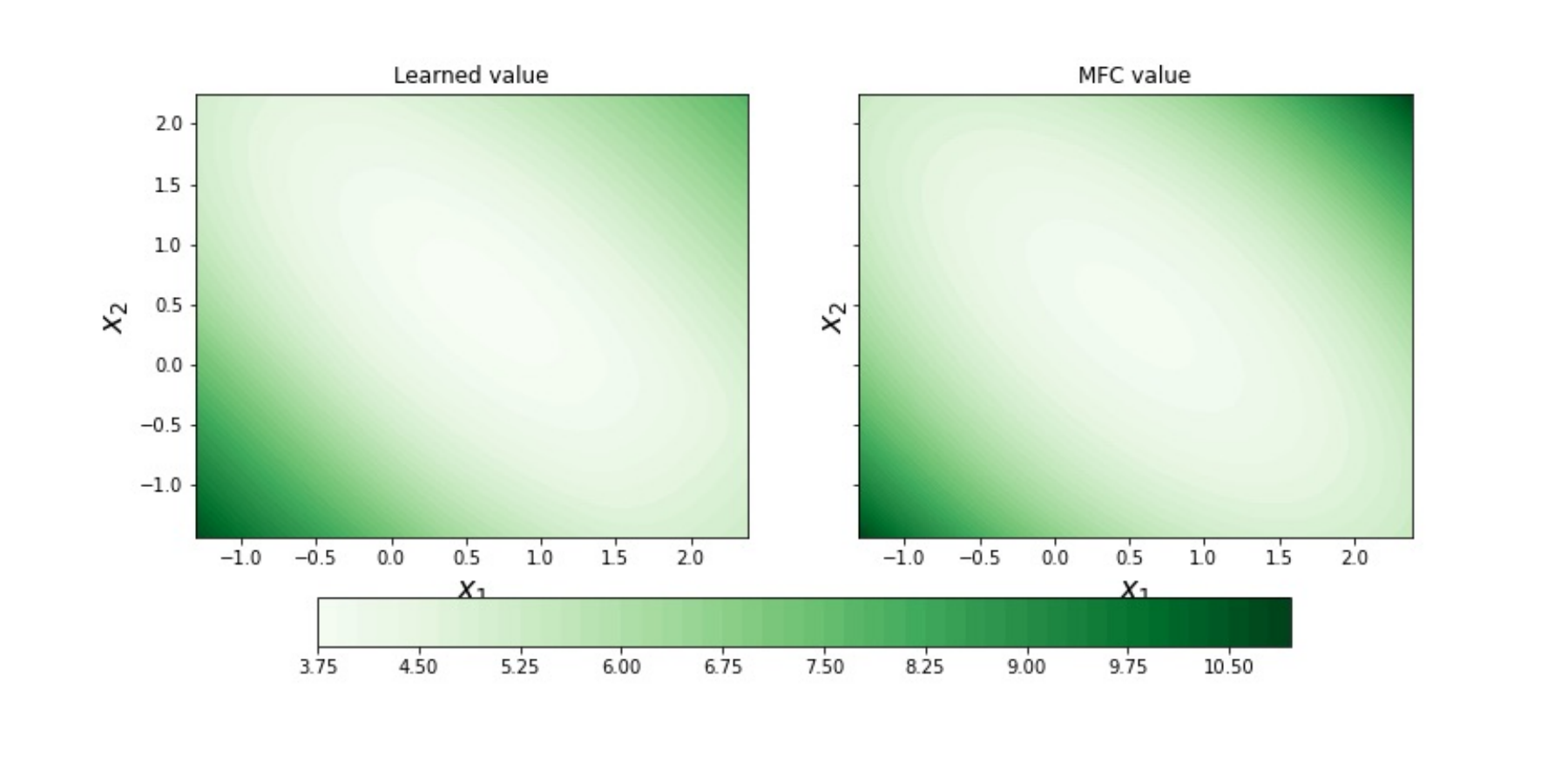}\label{fig:Ex_Im10}}
\caption{The right-hand side plots are the optimal value functions for the MFG and MFC problems, respectively. The left-hand side plots show the learned functions after $N=10^6$ iterations.}
\label{fig:5}
\end{figure}

The plots in Figure \ref{fig:5} show the temperature plots for the learned (left) and theoretical (right) value functions for MFG (top) and MFC (bottom). Similarly, we find that within the distribution supports in Figure \ref{fig:3}, the learned value function is close to the theoretical value function. Inconsistencies outside of the distribution supports will not influence the mean absolute error.

\section{Actor-critic Algorithm for Mean Field Control Games}\label{sec:MFCG}
As observed in \cite{10.1145/3533271.3561743}, in the case of tabular Q-learning, our IH-MFG-AC algorithm (Algorithm~\ref{algoMFG}) can easily be extended to the case of mixed mean field control game problems that involve two population distributions, a local one and a global one. This type of game corresponds to competitive games among a large number of large groups, where agents within each group collaborate. The local distribution represents the distribution within each ``representative'' agent's group, while the global distribution represents the distribution of the entire population.

Such games are motivated as follows: Consider a finite population of agents consisting of \( M \) groups, each containing \( N \) agents. An agent indexed by \( (m,n) \) indicates that she is the \( n \)-th member of the \( m \)-th group. Agents collaborate within their respective groups (sharing the same first index) and compete with all agents in other groups. In other words, all \( N \) agents in group \( m \) collectively work to minimize the total cost of group \( m \). The MFCG corresponds to the mean field limit as \( N \) and \( M \) tend to infinity, and serves as a proxy for the solution in the finite-player case. We refer to \cite{10.1145/3533271.3561743,JML-2-2, angiuli2024analysismultiscalereinforcementqlearning} for further details on MFCG, including the limit from finite-player games to infinite-player games. Note that the solution gives an approximation of the Nash equilibrium between the competitive groups.

The solution of an infinite-time horizon MFCG is a policy and population distribution pair
$
(\hat{\pi}, \hat{\mu}) \in \mathcal{A} \times \mathcal{P}(\mathbb{R}^d)
$
satisfying the following:

\begin{enumerate}
    \item \( \hat{\pi} \) solves the MFC problem
    \begin{align}
        \inf_{\pi \in \mathcal{A}} J_{\hat{\mu}}(\pi) = \inf_{\pi \in \mathcal{A}} \mathbb{E} \left[ \sum_{n=0}^\infty \gamma^n f\left(X_n^{\pi,\hat{\mu}}, \hat{\mu}, \mu^{\pi,\hat{\mu}}, A_n\right) \, \right],
        \tag{8.1}
        \label{8.1}
    \end{align}
    where the process $(X^{\pi,{\hat\mu}}_{n})_{n\geq 0}$ follows the dynamics
\begin{align*}
    \begin{cases}
        X_0^{\pi,\hat{\mu}} \sim \mu_0,
        \\
        P(X^{\pi,{\hat\mu}}_{n+1} \in \mathcal{B} |X^{\pi,{\hat\mu}}_{n}=x,\mu=\hat\mu,\mu'=\mu^{\pi,\hat\mu},A_{n}=a ) 
    = \int_{x'\in \mathcal{B}  } p(x,x',{\hat\mu},\mu^{\pi,\hat\mu},a)dx',\\
     A_{n} \sim \pi(\cdot|X_{n}^{\pi,\hat\mu})\quad \mbox{independently at each time}\,\, n\geq 0.
    \end{cases}
\end{align*}
    and 
    \[
    \mu^{\pi,\hat{\mu}} = \lim_{n \to \infty} \mathcal{L}\left(X_n^{\pi,\hat{\mu}}\right).
    \]
    
    \item Fixed point condition: \( \hat{\mu} = \lim_{n \to \infty} \mathcal{L}(X_n^{\hat{\pi}, \hat{\mu}}) \).
\end{enumerate}

Note that conditions 1 and 2 above imply that \( \hat{\mu} = \mu^{\hat{\pi}, \hat{\mu}} \). 

We can easily modify our MFC algorithm (Algorithm~\ref{algo:MFC}) by adding a slowly evolving measure to get a full algorithm as shown in Algorithm~\ref{algo:MFCG}.
\begin{algorithm}[H]
\caption{Infinite Horizon Mean Field Control Game Actor-Critic (IH-MFCG-AC) \label{algo:MFCG}} 
\begin{algorithmic}[1]
\REQUIRE 
Initial distribution $\mu^i$ and $\mu$; number of time steps $N$;  time-dependent neural network learning rates: for actor $\rho_n^\Pi$, and for critic: $\rho_n^V$; time-dependent learning rates $\rho_n$ for the population distribution; discounting coefficient $\gamma<1$
\STATE \textbf{Initialize neural network:} 
Actor $\Pi_{\psi_0}$ : $\mathbbm{R}^d\to\mathcal{P}(\mathbbm{R}^k)$, Critics $V_{\theta^i_0}$ : $\mathbbm{R}^d\to\mathbbm{R}$

\STATE \textbf{Sample} $X^i_{0}\sim \mu^i$, $X_{0}\sim \mu$ and save $\mu^i_0$ ,$\mu_0$, $\tilde\mu_0$ as $\delta_{X^i_0}$ ,$\delta_{X_0}$, $\delta_{X_0}$  \\
\FOR{$n=0,1,2,\dots, N-1$}
\IF{$X^i_{n}\in \proj_{\mathcal{X}}B^i$}
\STATE \textbf{Choose action:}  $A^{i}_{n}=a^i$
\STATE
\textbf{Observe reward from the environment:} $r^{i}_{n+1} = -f(X^i_{n},\mu^i_n,\tilde\mu_n,a^i)$
\STATE
\textbf{Observe the next state from the environment:} 
$
    X^{i}_{n+1} \sim p(X^i_{n},\cdot,\mu^i_{n},\tilde\mu_n, a^i)
$

\STATE
\textbf{Find Minimum Value at $X^{i}_{n+1}$:} $V_{n}(X^{i}_{n+1}) = \min_{j}V_{\theta^j_n}(X^{i}_{n+1})$ for $j$ such that $X^{i}_{n+1}\in \proj_{\mathcal{X}}B^j$
\STATE
\textbf{Compute TD target:} $y^{i}_{n+1} = r^{i}_{n+1} + \gamma 
V_n(X^{i}_{n+1})$

\STATE
\textbf{Compute TD error:} $\delta_{\theta^i_n} = y^{i}_{n+1} - V_{\theta^i_n}(X^i_{n})$

\STATE
\textbf{Compute critic loss:} $L_V^{(n)}(\theta^i_n) =(\delta_{\theta^i_n})^2 $

\STATE
\textbf{Update critic with SGD:} $\theta^i_{n+1} = \theta^i_n - \rho^V_n\nabla_{\theta^i} L_V^{(n)}(\theta^i_n)$ (freeze the TD target and take gradient only on $V_{\theta^i_n}$)

\ELSE
\STATE \textbf{Sample action:}  $A^i_{n} \sim \Pi_{\psi_n}(\cdot|X^i_{n})$ 
\STATE \textbf{Observe the next state from the environment:} 
$
    X^{i}_{n+1} \sim p(X^i_{n},\cdot,\mu^i_{n},\tilde\mu_n, A_n^i)
$
\ENDIF
\STATE \textbf{Update measure of bins:} $\mu_{n+1}^i = \mu_{n}^i-\rho_n (\mu_{n}^i-\delta_{X^{i}_{n+1}})$(use empirical distribution and save it as a weighted list)

\STATE \textbf{Sample action for the individual path:}  $A_{n} \sim \Pi_{\psi_n}(\cdot|X_{n})$
\STATE
\textbf{Observe reward from the environment:} $r_{n+1} = -f(X_{n},\mu_{n},\tilde\mu^i_n,A_{n})$. 
\STATE
\textbf{Observe the next state from the environment:} 
\[
X_{n+1} \sim p(X_{n},\cdot,\mu_{n},\tilde\mu_n, A_{n})
\]
\STATE
\textbf{Find $V_{n}(X_{n})$ and $V_n(X_{n+1})$:} Find $V_{n}(X_{n})$ and $V_n(X_{n+1})$ by definition as follows 
\[
V_{n}(x) = \min_j V_{\theta^j_n}(x) \text{ for } j \text{ such that } x\in \proj_{\mathcal{X}}B^j
\]
\STATE
\textbf{Compute TD target:} $y_{n+1} = r_{n+1} + \gamma V_n(X_{n+1})$
\STATE
\textbf{Compute TD error:} $\delta_{\theta_n} = y_{n+1} - V_{n}(X_{n})$
\STATE
\textbf{Compute actor loss:} $L_\Pi^{(n)}(\psi_n) =-\delta_{\theta_n}\log \Pi_{\psi_n}(A_{n}|X_{n}) $

\STATE
\textbf{Update actor with SGD:} $\psi_{n+1} = \psi_n - \rho^\Pi_n \nabla_\psi L_\Pi^{(n)}(\psi_n)$
\STATE \textbf{Update fast measure (local distribution) for individual path:} $\mu_{n+1} = \mu_{n}-\rho_n(\mu_n- \delta_{X_{n+1}})$(use empirical distribution and save it as a weighted list)
\STATE \textbf{Update slow measure (global distribution) for individual path:} $\tilde\mu_{n+1} = \tilde\mu_{n}-\tilde{\rho}_n(\mu_n- \delta_{X_{n+1}})$(use empirical distribution and save it as a weighted list)

\ENDFOR
\STATE {\textbf{Return}} $(\Pi_{\psi_N},\mu_N)$
\end{algorithmic}
\end{algorithm}

\subsection{MFCG Linear-quadratic Benchmark}
We test Algorithm~\ref{algo:MFCG} on the following linear-quadratic MFCG. The individual agent wish to minimize
\begin{equation*}
	\mathbb{E} \left[ \int_0^\infty e^{-\beta t} \left( 
	\frac{1}{2} \alpha_t^2 + c_1 \left( X_t^{\alpha,\mu} - c_2 m \right)^2 + c_3 \left( X_t^{\alpha,\mu} - c_4 \right)^2 
	+ \tilde{c}_1 \left( X_t^{\alpha,\mu} - \tilde{c}_2 m^{\alpha,\mu} \right)^2 + \tilde{c}_5 (m^{\alpha,\mu})^2
	\right) dt \right]
\end{equation*}
subject to the dynamics
\begin{equation}\label{8.5}
	dX_t^{\alpha,\mu} = \alpha_t \, dt + \sigma \, dW_t, \quad t \in [0,\infty),
\end{equation}
where
\[
m = \int x \, d\mu(x), \quad m^{\alpha,\mu} = \int x \, d\mu^{\alpha,\mu}(x),
\]
and the fixed point condition
\[
m = \lim_{t \to \infty} \mathbb{E} \left( X_t^{\hat{\alpha},\mu} \right) = m^{\hat{\alpha},\mu},
\]
where \( \hat{\alpha} \) is the optimal action.

We present the analytic solution to the MFCG problem using notation consistent with the derivation in \cite{JML-2-2}. The value function is defined as
\begin{equation*}
v(x) := \inf_{\alpha \in \mathcal{A}} \mathbb{E} \left[ \int_0^\infty e^{-\beta t} \left( 
\frac{1}{2} \alpha_t^2 + c_1 \left( X_t^{\alpha,\mu} - c_2 m \right)^2 + c_3 \left( X_t^{\alpha,\mu} - c_4 \right)^2 
+ \tilde{c}_1 \left( X_t^{\alpha,\mu} - \tilde{c}_2 m^{\alpha,\mu} \right)^2 + \tilde{c}_5 (m^{\alpha,\mu})^2
\right) dt \ \bigg| \ X_0 = x \right]
\end{equation*}

The explicit formula \( v(x) = \Gamma_2 x^2 + \Gamma_1 x + \Gamma_0 \) can be derived as the solution to the Hamilton-Jacobi-Bellman equation where
\[
\Gamma_2 = \frac{-\beta + \sqrt{\beta^2 + 8(c_1 + c_3 + \tilde{c}_1)}}{4},
\]
\[
\Gamma_1 = -\frac{2 \Gamma_2 c_3 c_4}{c_1 (1 - c_2) + \tilde{c}_1 (1 - \tilde{c}_2)^2 + c_3 + \tilde{c}_5},
\]
\[
\Gamma_0 = \frac{c_1 c_2^2 m^2 + (\tilde{c}_1 \tilde{c}_2^2 + \tilde{c}_5)(m^{\alpha,\mu})^2 + \sigma^2 \Gamma_2 - \Gamma_1^2 / 2 + c_3 c_4^2}{\beta}.
\]

Then the optimal control for the MFCG is
\begin{equation}\label{8.7}
\hat{\alpha}(x) = - (2 \Gamma_2 x + \Gamma_1).
\end{equation}

Substituting \eqref{8.7} into \eqref{8.5} yields the Ornstein-Uhlenbeck process
\[
    dX_t = -(2\Gamma_2 X_t + \Gamma_1) dt + \sigma dW_t,
\]
whose limiting distribution is
\begin{equation*}
	\hat{\mu} = \mu^{\hat{\alpha}, \hat{\mu}} = \mathcal{N} \left( -\frac{\Gamma_1}{2\Gamma_2}, \frac{\sigma^2}{4\Gamma_2} \right).
\end{equation*}
We note that an equation for $\hat m$ and $m^{\hat\alpha,\hat\mu}$ that only depends on the running cost coefficients is 
\[
    m:=\hat m = m^{\hat\alpha,\hat\mu} = \frac{c3c4}{c1(1-c2)+\tilde{c}_1(1-\tilde{c}_2)^2+c3+\tilde{c}_5}.
\]

\subsection{Numerical Results}
The results of the MFCG Algorithm are presented in Figs.~\ref{fig:mfcg-histo} and~\ref{fig:mfcg-value}.
\begin{figure}%
\center 
\subfloat[Result for Local Distribution]
{\includegraphics[width=0.45\textwidth]{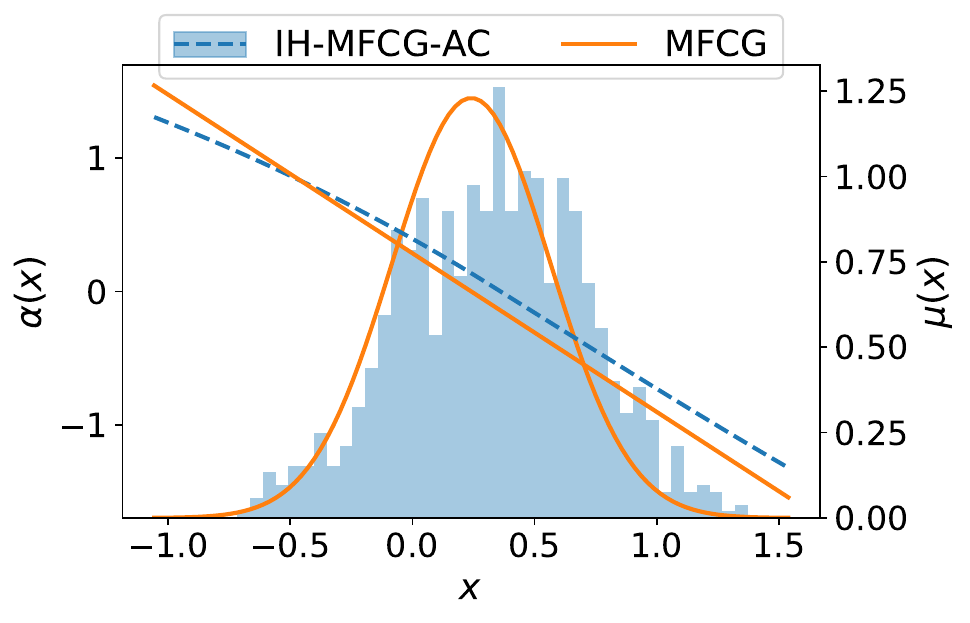}\label{fig:Ex_Im11}}
\subfloat[Result for Global Distribution]{\includegraphics[width=0.45\textwidth]{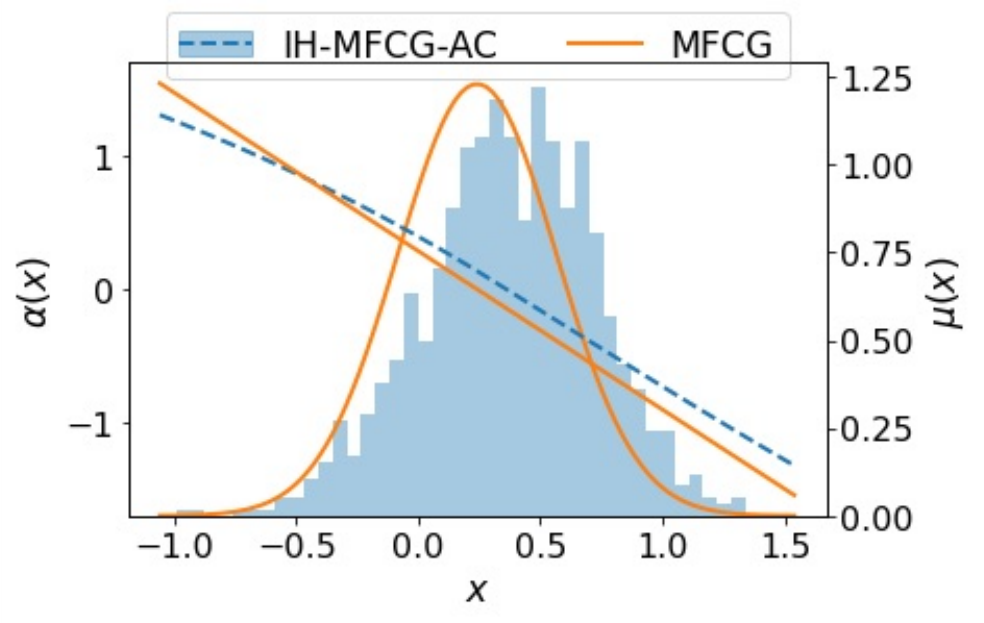}\label{fig:Ex_Im12}}
\caption{The histogram are the learned asymptotic distributions and the dashed line is the learned feedback control after $N=3\times10^6$ iterations. The orange curves correspond to the optimal control and mean field distribution for MFCG. The bottom axis shows the state variable $x$, the left axis refers to the value of the control $\alpha(x)$, and the right axis represents the probability density for $\mu(x)$.}
\label{fig:mfcg-histo}
\end{figure}

\begin{figure}%
\center 
\subfloat[Result for MFG]
{\includegraphics[width=0.45\textwidth]{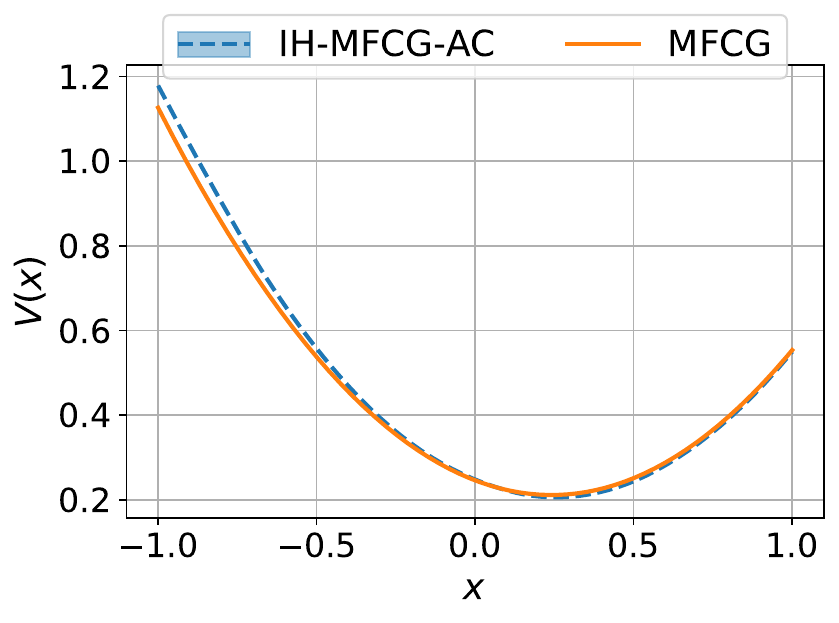}\label{fig:Ex_Im13}}
\caption{The orange curve is the optimal value function for the MFCG problem. The blue dashed line is the learned value function after $N = 3\times 10^6$ iterations. The light blue shaded region depicts one standard deviation from the learned value.}
\label{fig:mfcg-value}
\end{figure}

\section{Conclusion}
We analyze the deep actor–critic reinforcement learning algorithm introduced by \cite{angiuli2023deep} in the setting of continuous state and action spaces, discrete time, and an infinite time horizon. The method, generically called Infinite Horizon Mean Field Actor-Critic (IH-MF-AC), employs neural networks to parameterize both the policy and the value function, from which the optimal control is derived. Depending on the ratio of two learning rates—one for the value function and the other for the mean field term—the algorithm provides solutions to either Mean Field Game (MFG) or Mean Field Control (MFC) problems. We establish a rigorous convergence analysis and characterize the limiting points as MFG and MFC solutions, respectively. We then benchmark the method numerically on a linear–quadratic problem and show that IH-MF-AC closely recovers analytic solutions with high accuracy. In addition, we introduce a modified version, IH-MFCG-AC, designed to address Mean Field Control Game (MFCG) problems, and present preliminary results demonstrating its effectiveness. Promising directions for future research include extending the method to finite-horizon settings.

\bibliographystyle{unsrt} 
\bibliography{bibtex}

\begin{thebibliography}{10}

\bibitem{angiuli2023deep}
Andrea Angiuli, Jean-Pierre Fouque, Ruimeng Hu, and Alan Raydan.
\newblock Deep reinforcement learning for infinite horizon mean field problems
  in continuous spaces.
\newblock {\em arXiv preprint arXiv:2309.10953}, 2023.

\bibitem{Andrea23}
Andrea Angiuli, Jean-Pierre Fouque, Mathieu Lauri{\`e}re, and Mengrui Zhang.
\newblock Convergence of multiscale reinforcement {Q}-learning algorithms for
  mean field game and control problems.
\newblock {\em arXiv:2312.06659}, 2023.

\bibitem{Borkar97}
Vivek~S Borkar.
\newblock Stochastic approximation with two time scales.
\newblock {\em Systems \& Control Letters}, 29(5):291--294, 1997.

\bibitem{sutton2018reinforcement}
Richard~S Sutton and Andrew~G Barto.
\newblock {\em Reinforcement learning: An introduction}.
\newblock MIT press, 2018.

\bibitem{mnih2015human}
Volodymyr Mnih, Koray Kavukcuoglu, David Silver, Andrei~A Rusu, Joel Veness,
  Marc~G Bellemare, Alex Graves, Martin Riedmiller, Andreas~K Fidjeland, Georg
  Ostrovski, et~al.
\newblock Human-level control through deep reinforcement learning.
\newblock {\em nature}, 518(7540):529--533, 2015.

\bibitem{silver2016mastering}
David Silver, Aja Huang, Chris~J Maddison, Arthur Guez, Laurent Sifre, George
  Van Den~Driessche, Julian Schrittwieser, Ioannis Antonoglou, Veda
  Panneershelvam, Marc Lanctot, et~al.
\newblock Mastering the game of go with deep neural networks and tree search.
\newblock {\em nature}, 529(7587):484--489, 2016.

\bibitem{gu2017deep}
Shixiang Gu, Ethan Holly, Timothy Lillicrap, and Sergey Levine.
\newblock Deep reinforcement learning for robotic manipulation with
  asynchronous off-policy updates.
\newblock In {\em 2017 IEEE international conference on robotics and automation
  (ICRA)}, pages 3389--3396. IEEE, 2017.

\bibitem{vecerik2017leveraging}
Mel Vecerik, Todd Hester, Jonathan Scholz, Fumin Wang, Olivier Pietquin, Bilal
  Piot, Nicolas Heess, Thomas Roth{\"o}rl, Thomas Lampe, and Martin Riedmiller.
\newblock Leveraging demonstrations for deep reinforcement learning on robotics
  problems with sparse rewards.
\newblock {\em arXiv preprint arXiv:1707.08817}, 2017.

\bibitem{ouyang2022training}
Long Ouyang, Jeffrey Wu, Xu~Jiang, Diogo Almeida, Carroll Wainwright, Pamela
  Mishkin, Chong Zhang, Sandhini Agarwal, Katarina Slama, Alex Ray, et~al.
\newblock Training language models to follow instructions with human feedback.
\newblock {\em Advances in Neural Information Processing Systems},
  35:27730--27744, 2022.

\bibitem{busoniu2008comprehensive}
Lucian Busoniu, Robert Babuska, and Bart De~Schutter.
\newblock A comprehensive survey of multiagent reinforcement learning.
\newblock {\em IEEE Transactions on Systems, Man, and Cybernetics, Part C
  (Applications and Reviews)}, 38(2):156--172, 2008.

\bibitem{zhang2021multi}
Kaiqing Zhang, Zhuoran Yang, and Tamer Ba{\c{s}}ar.
\newblock Multi-agent reinforcement learning: A selective overview of theories
  and algorithms.
\newblock {\em Handbook of reinforcement learning and control}, pages 321--384,
  2021.

\bibitem{lanctot2017unified}
Marc Lanctot, Vinicius Zambaldi, Audrunas Gruslys, Angeliki Lazaridou, Karl
  Tuyls, Julien P{\'e}rolat, David Silver, and Thore Graepel.
\newblock A unified game-theoretic approach to multiagent reinforcement
  learning.
\newblock {\em Advances in neural information processing systems}, 30, 2017.

\bibitem{yang2020overview}
Yaodong Yang and Jun Wang.
\newblock An overview of multi-agent reinforcement learning from game
  theoretical perspective.
\newblock {\em arXiv preprint arXiv:2011.00583}, 2020.

\bibitem{MR2295621}
Jean-Michel Lasry and Pierre-Louis Lions.
\newblock Mean field games.
\newblock {\em Jpn. J. Math.}, 2(1):229--260, 2007.

\bibitem{MR2346927}
Minyi Huang, Roland~P. Malham{\'e}, and Peter~E. Caines.
\newblock Large population stochastic dynamic games: closed-loop
  {M}c{K}ean-{V}lasov systems and the {N}ash certainty equivalence principle.
\newblock {\em Commun. Inf. Syst.}, 6(3):221--251, 2006.

\bibitem{MR3134900}
Alain Bensoussan, Jens Frehse, and Sheung Chi~Phillip Yam.
\newblock {\em Mean field games and mean field type control theory}.
\newblock Springer Briefs in Mathematics. Springer, New York, 2013.

\bibitem{carmona2018probabilisticI-II}
Ren{\'e} Carmona and Fran{\c{c}}ois Delarue.
\newblock {\em Probabilistic Theory of Mean Field Games with Applications
  I-II}.
\newblock Springer, 2018.

\bibitem{lauriere2022learning}
Mathieu Lauri{\`e}re, Sarah Perrin, Julien P{\'e}rolat, Sertan Girgin, Paul
  Muller, Romuald {\'E}lie, Matthieu Geist, and Olivier Pietquin.
\newblock Learning mean field games: A survey.
\newblock {\em arXiv preprint arXiv:2205.12944}, 2022.

\bibitem{guo2019learning}
Xin Guo, Anran Hu, Renyuan Xu, and Junzi Zhang.
\newblock Learning mean-field games.
\newblock In {\em Advances in Neural Information Processing Systems}, pages
  4966--4976, 2019.

\bibitem{cui2021approximately}
Kai Cui and Heinz Koeppl.
\newblock Approximately solving mean field games via entropy-regularized deep
  reinforcement learning.
\newblock In {\em International Conference on Artificial Intelligence and
  Statistics}, pages 1909--1917. PMLR, 2021.

\bibitem{anahtarci2023q}
Berkay Anahtarci, Can~Deha Kariksiz, and Naci Saldi.
\newblock {Q}-learning in regularized mean-field games.
\newblock {\em Dynamic Games and Applications}, 13(1):89--117, 2023.

\bibitem{elie2020convergence}
Romuald Elie, Julien Perolat, Mathieu Lauri{\`e}re, Matthieu Geist, and Olivier
  Pietquin.
\newblock On the convergence of model free learning in mean field games.
\newblock In {\em in proc. of AAAI}, 2020.

\bibitem{perrin2020continuousfp}
Sarah Perrin, Julien P{\'e}rolat, Mathieu Lauri{\`e}re, Matthieu Geist, Romuald
  Elie, and Olivier Pietquin.
\newblock {Fictitious Play for Mean Field Games: Continuous Time Analysis and
  Applications}.
\newblock {\em Advances in neural information processing systems},
  33:13199--13213, 2020.

\bibitem{lauriere2022scalable}
Mathieu Lauriere, Sarah Perrin, Sertan Girgin, Paul Muller, Ayush Jain,
  Theophile Cabannes, Georgios Piliouras, Julien P{\'e}rolat, Romuald Elie, and
  Olivier Pietquin.
\newblock Scalable deep reinforcement learning algorithms for mean field games.
\newblock In {\em International conference on machine learning}, pages
  12078--12095. PMLR, 2022.

\bibitem{magnino2025solving}
Lorenzo Magnino, Kai Shao, Zida Wu, Jiacheng Shen, and Mathieu Lauri{\`e}re.
\newblock Solving continuous mean field games: Deep reinforcement learning for
  non-stationary dynamics.
\newblock In {\em Advances in Neural Information Processing Systems}, 2025.

\bibitem{mguni2018decentralised}
David Mguni, Joel Jennings, and Enrique~Munoz de~Cote.
\newblock Decentralised learning in systems with many, many strategic agents.
\newblock In {\em Thirty-Second AAAI Conference on Artificial Intelligence},
  2018.

\bibitem{SubramanianMahajan-2018-RLstatioMFG}
Jayakumar Subramanian and Aditya Mahajan.
\newblock Reinforcement learning in stationary mean-field games.
\newblock In {\em Proceedings. 18th International Conference on Autonomous
  Agents and Multiagent Systems}, 2019.

\bibitem{carmona2023model}
Ren{\'e} Carmona, Mathieu Lauri{\`e}re, and Zongjun Tan.
\newblock Model-free mean-field reinforcement learning: mean-field {MDP} and
  mean-field {Q}-learning.
\newblock {\em The Annals of Applied Probability}, 33(6B):5334--5381, 2023.

\bibitem{gu2021meanQ}
Haotian Gu, Xin Guo, Xiaoli Wei, and Renyuan Xu.
\newblock Mean-field controls with {Q}-learning for cooperative {MARL}:
  convergence and complexity analysis.
\newblock {\em SIAM Journal on Mathematics of Data Science}, 3(4):1168--1196,
  2021.

\bibitem{CarmonaLauriereTan-2019-LQMFRL}
Ren{\'e} Carmona, Mathieu Lauri{\`e}re, and Zongjun Tan.
\newblock Linear-quadratic mean-field reinforcement learning: Convergence of
  policy gradient methods.
\newblock {\em arXiv preprint arXiv:1910.04295}, 2019.

\bibitem{frikha2025actor}
Noufel Frikha, Maximilien Germain, Mathieu Lauri{\`e}re, Huy{\^e}n Pham, and
  Xuanye Song.
\newblock Actor-critic learning for mean-field control in continuous time.
\newblock {\em Journal of Machine Learning Research}, 26(127):1--42, 2025.

\bibitem{pasztor2021efficientmodelbased}
Ilija Bogunovic.
\newblock Efficient model-based multi-agent mean-field reinforcement learning.
\newblock {\em Transactions on Machine Learning Research}, 5:1--35, 2023.

\bibitem{Andrea20}
Andrea Angiuli, Jean-Pierre Fouque, and Mathieu Lauri{\`e}re.
\newblock Unified reinforcement {Q}-learning for mean field game and control
  problems.
\newblock {\em Mathematics of Control, Signals, and Systems}, 34:217--271,
  2022.

\bibitem{10.1145/3533271.3561743}
Andrea Angiuli, Nils Detering, Jean-Pierre Fouque, Mathieu Lauri\`{e}re, and
  Jimin Lin.
\newblock Reinforcement learning for intra-and-inter-bank borrowing and lending
  mean field control game.
\newblock In {\em Proceedings of the Third ACM International Conference on AI
  in Finance}, ICAIF '22, page 369–376, New York, NY, USA, 2022. Association
  for Computing Machinery.

\bibitem{carmona2013control}
Ren{\'e} Carmona, Fran{\c{c}}ois Delarue, and Aim{\'e} Lachapelle.
\newblock Control of {M}c{K}ean--{V}lasov dynamics versus mean field games.
\newblock {\em Mathematics and Financial Economics}, 7:131--166, 2013.

\bibitem{carmona2015mean}
Rene Carmona, Jean~Pierre Fouque, and Li~Hsien Sun.
\newblock Mean field games and systemic risk.
\newblock {\em Communications in Mathematical Sciences}, 13(4):911--933, 2015.

\bibitem{carmona2019price}
Ren{\'e} Carmona, Christy~V Graves, and Zongjun Tan.
\newblock Price of anarchy for mean field games.
\newblock {\em ESAIM: Proceedings and Surveys}, 65:349--383, 2019.

\bibitem{cardaliaguet2019efficiency}
Pierre Cardaliaguet and Catherine Rainer.
\newblock On the (in) efficiency of {MFG} equilibria.
\newblock {\em SIAM Journal on Control and Optimization}, 57(4):2292--2314,
  2019.

\bibitem{carmona2023nash}
Rene Carmona, G{\"o}k{\c{c}}e Dayan{\i}kl{\i}, Francois Delarue, and Mathieu
  Lauri{\`e}re.
\newblock From {N}ash equilibrium to social optimum and vice versa: a mean
  field perspective.
\newblock arXiv preprint arXiv:2312.10526, 2023.

\bibitem{shao2024deepseekmathpushinglimitsmathematical}
Zhihong Shao, Peiyi Wang, Qihao Zhu, Runxin Xu, Junxiao Song, Xiao Bi, Haowei
  Zhang, Mingchuan Zhang, Y.~K. Li, Y.~Wu, and Daya Guo.
\newblock Deepseekmath: Pushing the limits of mathematical reasoning in open
  language models, 2024.

\bibitem{doi:10.1137/S0040585X97986825}
Oleg~A. Butkovsky.
\newblock On ergodic properties of nonlinear markov chains and stochastic
  mckean--vlasov equations.
\newblock {\em Theory of Probability \& Its Applications}, 58(4):661--674,
  2014.

\bibitem{NIPS2007_6883966f}
Shalabh Bhatnagar, Mohammad Ghavamzadeh, Mark Lee, and Richard~S Sutton.
\newblock Incremental natural actor-critic algorithms.
\newblock In J.~Platt, D.~Koller, Y.~Singer, and S.~Roweis, editors, {\em
  Advances in Neural Information Processing Systems}, volume~20. Curran
  Associates, Inc., 2007.

\bibitem{NIPS1999_6449f44a}
Vijay Konda and John Tsitsiklis.
\newblock Actor-critic algorithms.
\newblock In S.~Solla, T.~Leen, and K.~M\"{u}ller, editors, {\em Advances in
  Neural Information Processing Systems}, volume~12. MIT Press, 1999.

\bibitem{SELL197342}
George~R. Sell.
\newblock Differential equations without uniqueness and classical topological
  dynamics.
\newblock {\em Journal of Differential Equations}, 14(1):42--56, 1973.

\bibitem{563625}
Vivek~S. Borkar and Krishnamurthy Soumyanatha.
\newblock An analog scheme for fixed point computation. {I}. {T}heory.
\newblock {\em IEEE Transactions on Circuits and Systems I: Fundamental Theory
  and Applications}, 44(4):351--355, 1997.

\bibitem{doi:10.1287/opre.2022.2395}
Haotian Gu, Xin Guo, Xiaoli Wei, and Renyuan Xu.
\newblock Dynamic programming principles for mean-field controls with learning.
\newblock {\em Operations Research}, 71(4):1040--1054, 2023.

\bibitem{wang2020reinforcement}
Haoran Wang, Thaleia Zariphopoulou, and Xun~Yu Zhou.
\newblock Reinforcement learning in continuous time and space: A stochastic
  control approach.
\newblock {\em Journal of Machine Learning Research}, 21(198):1--34, 2020.

\bibitem{JML-2-2}
Andrea Angiuli, Detering Nils, Fouque Jean-Pierre, Laurière Mathieu, and Jimin
  Lin.
\newblock Reinforcement learning algorithm for mixed mean field control games.
\newblock {\em Journal of Machine Learning}, 2(2):108--137, 2023.

\bibitem{angiuli2024analysismultiscalereinforcementqlearning}
Andrea Angiuli, Jean-Pierre Fouque, Mathieu Laurière, and Mengrui Zhang.
\newblock Analysis of multiscale reinforcement {Q}-learning algorithms for mean
  field control games.
\newblock {\em arXiv:2405.17017}, 2024.

\end{thebibliography}
\end{document}